\documentclass[11pt,a4paper]{amsart}
\usepackage{graphicx}
\usepackage{amssymb}
\usepackage[all]{xy}
\usepackage{amsmath}
\usepackage[french,english]{babel}
\usepackage{amssymb,amsmath,amsfonts,amsthm,amscd}
\usepackage{indentfirst,graphicx}
\usepackage[font={scriptsize},captionskip=5pt]{subfig}

\def\Projan{\mathop{\rm Projan}}
\newtheorem{theorem}{Theorem}[section]
\newtheorem{exa}[theorem]{Example}

\newtheorem{rem}[theorem]{Remark}

\newtheorem{cor}[theorem]{Corollary}

\newtheorem{tont}[theorem]{Definition}
\newenvironment{definition}{\begin{tont} \em}{\end{tont}}
\newtheorem{lem}[theorem]{Lemma}
\newenvironment{Lemme}{\begin{lem} \em}{\end{lem}}

\newtheorem{Prop}[theorem]{Proposition}
 
\let\:=\colon
\let\To=\longrightarrow
\newcommand{\Mat}{Mat}

\newcommand{\cO}{{\mathcal O}}

\newcommand{\cR}{{\mathcal R}}

\newcommand{\cX}{{\mathcal X}}

\def\Projan{\mathop{\rm Projan}}
\def\Af{\hbox{\rm A$_f$\hskip 2pt}}
\def\AF{\hbox{\rm A$_F$\hskip 2pt}}

\def\Wf{\hbox{\rm W$_f$\hskip 2pt}}
\def\WF{\hbox{\rm W$_F$\hskip 2pt}}

\begin{document}

\title [Determinantal singularities]{Pairs of modules and determinantal isolated singularities }

\author[T. Gaffney]{Terence Gaffney}\thanks{T.~Gaffney was partially supported by PVE-CNPq Proc. 401565/2014-9}
 \address{T. Gaffney, Department of Mathematics\\
  Northeastern University\\
  Boston, MA 02215}

\author[A. Rangachev]{Antoni Rangachev}
\address {A. Rangachev, Department of Mathematics\\
  Northeastern University\\
  Boston, MA 02215 \\
and Department of Analysis, Geometry and Topology, Institute of Mathematics \\
and Informatics, Bulgarian Academy of Sciences\\
Akad. G. Bonchev, Sofia 1113, Bulgaria}

\begin{abstract}
We continue the development of the study of the equisingularity of isolated singularities, in the determinantal case.
\end{abstract}

\maketitle

\selectlanguage{english}

\section{Introduction}

This paper studies the Whitney equisingularity of  families of determinantal singularities. In the study of equisingularity of sets we are given a family of sets (or mappings) and we want to find invariants which depend only on the members of the family, whose independence of parameter ensures that an equisingularity condition holds. Successful examples include hypersurfaces (\cite {T_C}) and complete intersections with isolated singularities (\cite{G-2}). Smoothable determinantal singularities include these two classes of singularities and can be regarded as the next class in order of complexity, as we shall see.

In the case of isolated hypersurface and complete intersection singularities, the invariant which is related to the infinitesimal geometry of the set has a geometric interpretation in terms of the generic element to which the object deforms. In the case of germs of functions with an isolated critical point,  the generic element is a function with only Morse critical points, and the number of such points is the Milnor number--the multiplicity of the Jacobian ideal. In turn, this number is the degree of the exceptional divisor of the blowup of the ambient space by the Jacobian ideal. 

If the object is a hypersurface $X$ with an isolated singularity then the generic element is a smoothing, and the basic infinitesimal invariant is sum of the Milnor number of $X$ and of a generic hyperplane slice of $X$. This is also the multiplicity of the Jacobian ideal in the local ring of $X,0$.

In the study of determinantal singularities, the desire for invariants depending only on the members of the family and the desire to preserve the connection between the invariants of the infinitesimal geometry and those related to the geometry of the generic element to which the singularity deforms, leads us to introduce the notion of the {\it landscape} of a singularity. 

Choosing the  landscape of a singularity $X$ consists of defining the allowable families that include the set, and its generic perturbations. Each set should have a unique generic element that it deforms to.  There should exist a connection between invariants related to  the infinitesimal geometry of $X$, and some elements of the topology of this generic element.  Describing the connection between  the infinitesimal geometry of $X$ and the topology of the generic element related to $X$ is part of understanding the landscape. 

In studying the equisingularity of a family of isolated singularities, choosing the landscape can be done by fixing in advance a component of the base space of the versal deformation, from which the given family is induced. This can be done explicitly or implicitly as in our first case where we restrict to determinantal deformations. This has the effect of fixing the generic fiber of the versal deformation to which all members of our family can be deformed. Invariants associated with the geometry/topology of this general member  provide important information about the singularities in the original family.

As the example in section 6 shows, for the same singularity $X$, there can be two different choices of landscape. In the example, for each choice of landscape there is a Whitney equisingular family which contains $X$. The invariants of $X$ which control the Whitney equisingularity of the family depend on the choice of landscape.  Each choice of landscape gives a different generic object to which $X$ deforms, and their differing topology accounts for the differing infinitesimal invariants.

In an earlier paper the first author introduced a framework for studying the equisingularity of families of isolated singularities using the multiplicity of pairs of modules and their polar curves (\cite{Gaff1}). The Jacobian module of $X$, $JM(X)$, which is the module generated by the partial derivatives of a set of defining equations for $X$, is one of the elements of the pair as it relates well to the Whitney conditions. 

The choice of  allowable deformations determines the corresponding first order infinitesimal deformations of  $X$. These make up a module $N(X)$, which is the second and larger module in the pair.   For the case of sets, the invariants we need for checking condition W come from the pair $(JM(X), N(X))$ and $N(X)$ by itself. A change at the infinitesimal level of the family is always tied to a change in topology of the generic related elements. 

 Recall that a determinantal singularity is one whose ideal is generated by the minors of a  matrix, where the singularity so defined has the expected dimension. A determinantal singularity is a maximal rank singularity if the order of the minors is maximal for the given matrix. This type of singularity contains all ICIS singularities, and is the type studied in this paper. The allowable deformations of $X$ are those induced by a deformation of the presentation matrix of $X$. We choose the size of the matrix and the dimension of the ambient space ${\mathbb C}^q$ so that $X$ is smoothable, hence the generic element to which $X$ deforms is a smooth manifold.

Because they are given by a presentation matrix, we can add some more structure to the study of determinantal singularities.  Determinantal singularities are sections of stable singularities.  We say a singularity $S$ is {\it stable} if the Jacobian module of $S$ agrees with its module of infinitesimal deformations. We show that  the singularities $\overline{\Sigma}_r$, the matrices in $Hom ({\mathbb C}^n, {\mathbb C}^{n+k})$ of kernel rank $r$ or more are stable. 

A determinantal singularity $X$, with presentation matrix $M_X$, can be viewed as the intersection of the graph of $M_X$, seen as a map from ${\mathbb C}^q$ to $Hom ({\mathbb C}^n, {\mathbb C}^{n+k})$, with ${\mathbb C}^q\times \overline{\Sigma}_r$. It is in this sense that they are sections of stable singularities.

We are interested in the case where the section has a smoothing by varying the section of $S$. The invariant of interest from our framework for equisingularity purposes, is the sum of multiplicity of the pair of modules--Jacobian module of $X$ and $N(X)$, the pullback of the infinitesimal deformations of $S$ and the multiplicity of the polar curve of $N(\cX)$, $\cX$ a deformation of $X$ to a smoothing.  Because the larger module in our pair is the pullback of a module on $S$, it is {\it universal}. This means that $N(\cX)$ specializes in families, avoiding the problem that occurs for the choice of $N$ in \cite{Gaff1}, for isolated singularities whose versal deformation space does not have a smooth base.

In the case of ICIS singularities (\cite{G-4}), because we can freely deform the defining equations of our space, $N(X)$ is free, and the multiplicity of the pair becomes the Buchsbaum-Rim multiplicity of the Jacobian module.   Again, because $N(X)$ is free, it has no polar varieties. In the determinantal case, we cannot deform the equations of $X$ freely, however we can deform the entries of the presentation matrix $M_X$ freely. This means that  $N(X)$ is not free and in general, for determinantal singularities, $\Projan {\mathcal R}N(X)$ has nontrivial geometry which enters into the invariants.  This change in the nature of $N(X)$ is the main reason this class of singularities is more complex.

 Combining Theorem 2.5 and Corollary 3.4, we show that the multiplicity of the polar curve of $N(\cX)$ is the intersection number of the image of the section with the polar variety of $\overline{\Sigma}_r$, of complementary dimension. We compute formulas for the polar varieties of $\overline{\Sigma}_r$ for the case of maximal rank singularities. This in turn computes the  multiplicity of the polar curve of $N(\cX)$ in terms of the presentation matrix of $X$. We then use the framework to study equisingularity problems for these singularities. The framework yields invariants which control the Whitney equisingularity type whose definition is independent of the family under study. These invariants are also connected with the topology of the smoothing.

Since the polar varieties of $S$ are important to us, and they are defined using the conormal variety of $S$,  as part of the development of this setting, we give a description of the conormal variety of the stable singularities of which determinantal varieties are sections. These are the sets $\overline{\Sigma}_r$, the matrices in $Hom ({\mathbb C}^n, {\mathbb C}^{n+k})$ of kernel rank $r$ or more. This enables us to show easily that the intersection numbers we need are zero in a range of dimensions because the complementary polars are empty. In this case our invariants simplify to $e(JM(X),N(X))$.

The paper is organized as follows.

In the second section we introduce background material from theory of integral closure of modules, and describe the framework developed in  (\cite{Gaff1}) for proving equisingularity results. We introduce the notion of a stable singularity, and consider the landscape given by sections $M$ of these singularities. We compute the multiplicity of the polar varieties of $N(X)$ as an intersection number of the polar varieties of the stable singularity with the map defining $X$. We prove a result about the conormal spaces of the $\overline{\Sigma}_r$ singularities. We also prove a result showing that the polar variety of codimension $d$ in a family of $d$-dimensional spaces can be controlled using the codimension $d$ polar varieties in a larger family. 

 In the third section we look at this framework for sections of determinantal singularities and introduce the pair of modules we use to control equisingularity. One is the Jacobian module JM(X) of the singularity $X$, the other is the first order infinitesimal determinantal deformations of the singularity. We denote this module by $N(X)$ or $N$ when $X$ is understood.  

To calculate the multiplicity of the polar curve of $N$ in a deformation to a smoothing, it is necessary describe $\Projan \cR(N)$. In the maximal rank case, we do this by showing an equivalence between $\Projan \cR(N)$ and a modification of $X$ based on the presentation matrix of the singularity. This equivalence then gives a decomposition of the multiplicity of the polar of $N$ as a sum of intersection numbers of generic plane sections with the exceptional fiber of the modification. This means we
can compute the  intersection number of the image of the section with the polar variety of $S$ of complementary dimension, as a sum of intersection numbers of modules naturally associated with the singularity; these intersection numbers in turn are the colengths of a collection of ideals. Using this, we give a formula for the codimension $d$ polar multiplicity of the singularity in terms of the multiplicity of the pair and an alternating sum of colengths of ideals.

 In section four we compute these intersection numbers as the alternating sum of intersections of modules which depend only on the presentation matrix, and give an example of a computation for a family of space curves. 

In section 5 we review various equisingularity conditions and describe the consequences for these conditions based on section 4. The framework of section three joined with the results of section 4 give conditions for Whitney equisingularity, and the relative conditions $\Af$ and $\Wf$, which depend only on the presentation matrix of the singularity, and the multiplicity of the  pair of modules corresponding to the condition. 

Section 6  contains the example of a singularity which is a member of two Whitney equisingular families, whose generic elements have topologically distinct smoothings. This example shows that it is impossible to find an invariant which depends only on an analytic space $X$ with an isolated singularity, whose value is independent of parameter for all Whitney equisingular deformations of $X$, and which is determined by the geometry of a smoothing of $X$.

The authors are happy to acknowledge helpful conversations with Steven Kleiman, whose paper \cite{K2} provides a useful result for calculating the multiplicity of a pair of modules in the case of curves, and to Anne Fr\"ubis-Kr\"uger for some background on Tjurina transforms.

\section{The theory of the integral closure of modules, polar varieties and conormal spaces }

Let  $(X, x)$ be a germ of a complex analytic space and $X$ a
small representative of the germ, and let $\mathcal{O}_{X}$ denote the
structure sheaf on a complex analytic space $X$. For simplicity we assume $X$ is equidimensional, and if $M$ is a sheaf of modules on $X$, $M$  a subsheaf of a free sheaf $F$, then $g$ is the generic rank of $M$ on each component of $X$. If we fix a set of generators of a module $M$, then we write $[M]$ for the matrix of generators.

\begin{definition} Suppose $(X, x)$ is the germ of a complex analytic space,
$M$ a submodule of $\mathcal{O}_{X,x}^{p}$. Then $h \in
\mathcal{O}_{X,x}^{p}$ is in the integral closure of $M$, denoted
$\overline{M}$, if for all analytic $\phi : (\mathbb{C}, 0) \to (X,
x)$, $h \circ \phi \in (\phi^{*}M)\mathcal{O}_{1}$. If $M$ is a
submodule of $N$ and $\overline{M} = \overline{N}$ we say that $M$
is a reduction of $N$.
\end{definition}

To check the definition it suffices to check along a finite number of curves whose generic point is in the Zariski open subset of $X$ along which $M$ has maximal rank. (Cf. \cite {G-2}.)

If a module $M$ has finite colength in $\mathcal{O}_{X,x}^{p}$, it
is possible to attach a number to the module, its Buchsbaum-Rim
multiplicity,  $e(M,\mathcal{O}_{X,x}^{p}).$ We can also define the multiplicity $e(M,N)$ of a pair of
modules $M \subset N$, $M$ of finite colength in $N$, as well, even
if $N$ does not have finite colength in $\mathcal{O}_{X}^{p}$.

We recall how to construct the multiplicity of a pair of modules using the approach of
Kleiman and Thorup \cite{KT}. Given a submodule $M$ of a free
$\mathcal{O}_{X^{d}}$ module $F$ of rank $p$, we can associate a
subalgebra $\mathcal{R}(M)$ of the symmetric $\mathcal{O}_{X^{d}}$
algebra on $p$ generators. This is known as the Rees algebra of $M$.
If $(m_{1},\ldots ,m_{p})$ is an element of $M$ then $\sum
m_{i}T_{i}$ is the corresponding element of $\mathcal{R}(M)$. Then
$\Projan(\mathcal{R}(M))$, the projective analytic spectrum of
$\mathcal{R}(M)$ is the closure of the projectivized row spaces of
$M$ at points where the rank of a matrix of generators of $M$ is
maximal. Denote the projection to $X^{d}$ by $c$. If $M$ is a
submodule of $N$ or $h$ is a section of $N$, then $h$ and $M$
generate ideals on $\Projan \mathcal{R}(N)$; denote them by $\rho(h)$
and $\rho(\mathcal{M})$. If we can express $h$ in terms of a set of
generators $\{n_{i}\}$ of $N$ as $\sum g_{i}n_{i},$ then in the
chart in which $T_{1}\neq 0,$ we can express a generator of
$\rho(h)$ by $\sum g_{i}T_{i}/T_{1}.$ Having defined the ideal sheaf
$\rho(\mathcal{M}),$ we blow it up.

On the blow up $B_{\rho(\mathcal{M})}(\Projan \mathcal{R}(N))$ we have
two tautological bundles. One is the pullback of the bundle on
$\Projan \mathcal{R}(N)$. The other comes from $\Projan
\mathcal{R}(M)$. Denote the corresponding Chern classes by $c_{M}$
and $c_{N}$, and denote the exceptional divisor by $D_{M,N}$.
Suppose the generic rank of $N$ on every component of $X$ is also $g$.

Then the multiplicity of a pair of modules $M, N$ is:

$$
e(M,N) = \sum_{j=0}^{d+g-2}\int D_{M,N}\cdot c_{M}^{d+g-2-j}\cdot
c_{N}^{j}.
$$

Kleiman and Thorup show that this multiplicity is well defined at $x
\in X$ as long as $\overline{M} = \overline{N}$ on a deleted
neighborhood of $x$. This condition implies that $D_{M,N}$ lies in
the fiber over $x$, hence is compact. Notice that when $N=F$ and $M$ has finite colength in $F$ then $e(M,N)$ is the Buchsbaum-Rim multiplicity $e(M,\mathcal{O}_{X,x}^{p})$. There is a fundamental result due to Kleiman and Thorup, the principle of additivity \cite{KT}, which states that given a sequence of $\mathcal{O}_{X,x}$-modules $M\subset N \subset P$  such that the multiplicity of the pairs is well defined, then$$e(M,P)=e(M,N)+e(N,P).$$ Also if $\overline{M}=\overline{N}$ then $e(M,N)=0$ and the converse also holds if $X$ is equidimensional. Combining these two results we get that  $\overline{M}=\overline{N}$ if and only if  $e(M,N)=e(N,P).$ These results will be used in Section 5.

In studying the geometry of singular spaces, it is natural to study
pairs of modules. In dealing with non-isolated singularities, the
modules that describe the geometry have non-finite colength, so
their multiplicity is not defined. Instead, it is possible to define
a decreasing sequence of modules, each with finite colength inside
its predecessor, when restricted to a suitable complementary plane.
Each pair controls the geometry in a particular codimension.

We also need the notion of the polar varieties of $M$. The {\it polar variety of codimension $l$} of $M$ in $X$, denoted
$\Gamma_l(M)$, is constructed by intersecting $\Projan{\mathcal R}(M)$
 with $X\times H_{g+l-1}$ where
$H_{g+l-1}$ is a general plane of codimension $g+l-1$, then projecting to
$X$.

The polar varieties of $M$ can be constructed by working only on $X$. The plane $H_{g+l-1}$ consists of all hyperplanes containing a fixed plane $H_K$ of dimension $g+l-1$; by multiplying the matrix of generators of $M$ by a basis of $H_K$ we obtain a submodule of $M$ denoted $M_H$.

\begin {Prop} In this set-up the polar variety of codimension $l$ consists of the closure in $X$ of the set of points where the rank of $M_H$ is less than $g$, and the rank of $M$ is $g$.\end{Prop}
\begin{proof} Since $H_{g+l-1}$ is generic, the general point of $\Projan{\mathcal R}(M)\cap X\times H_{g+l-1}$ lies over points where the rank of $M$ is $g$. Choose coordinates so that a basis for $H_K$ consists of the last $g+l-1$ elements of the standard basis of $\mathbb C^j$, $j$ the number of generators of $M$. We can find $v$ such that $v[M_H]=0$ but $v[M]\ne 0$ if and only if we are at a point where the rank of $M_H<g$. The existence of $v$ is equivalent to being able to find a combination of the rows of $[M]$, such that the last $g+l-1$ entries are $0$. This row is a hyperplane which lies in $H_{g+l-1}$.
\end{proof}

Setup: We suppose we have families  of modules $M\subset  N$, $M$ and $N$
 submodules of a free module $F$ of rank $p$
 on an equidimensional family of spaces with equidimensional
 fibers ${\mathcal X}^{d+k}$, ${\mathcal X}$ a family over a smooth base
$Y^k$. We assume that the generic rank of $M$, $N$ is $g\le p$ on every component of the fibers.  Let
$P(M)$ denote $\Projan {\mathcal R}(M)$, $\pi_M$
 the projection to ${\mathcal X}$.

We will be interested in computing, as we move from the special point $0$ to a generic point, the change in the multiplicity of
the pair $(M,N),$ denoted $\Delta(e(M,N))$. We will assume that the integral closures of $M$ and $N$ agree off a set $C$ of dimension
$k$ which is finite over $Y$, and assume we are working on a
sufficiently small neighborhood of the origin, so that every component
of $C$ contains the origin in its closure. Then $e(M,N, y)$ is the
sum of the multiplicities of the pair at all points in the fiber of
$C$ over $y$, and $\Delta(e(M,N))$ is the change in this number from
$0$ to a generic value of $y.$ If we have a set $S$ which is finite
over $Y$, then we can project $S$ to $Y$, and the degree of the
branched cover at $0$ is $mult_{y} S.$ (Of course, this is just the
number of points in the fiber of $S$ over our generic $y.$)

Let $C(M)$ denote the locus of points where $M$ is not free, {\it i.e.}, the
points where the rank of $M$ is less
than $g$, $C(\Projan {\mathcal R}(M))$
its inverse image under $\pi_M$.

We can now state the Multiplicity Polar Theorem. The proof in the ideal case appears in \cite{Gaff1}; the general proof appears in \cite{Gaff}.

\begin{theorem}(Multiplicity Polar Theorem) Suppose in the above
setup we have that $\overline{M} = \overline{N}$ off a set $C$ of
dimension $k$ which is finite over $Y$. Suppose further that
$C(\Projan\mathcal{R}(M))(0) = C(\Projan\mathcal{R}(M(0)))$ except
possibly at the points which project to $0 \in \mathcal{X}(0).$
Then, for y a generic point of $Y$, $$\Delta(e(M,N)) =
mult_{y}\Gamma_{d}(M) - mult_{y}\Gamma_{d}(N)$$
where ${\mathcal X}(0)$ is the fiber over $0$ of the family ${\mathcal X}^{d+k}$, $C(\Projan\mathcal{R}(M))(0)$ is the fiber of $C(\Projan\mathcal{R}(M))$ over $0$ and $M(0)$ is the restriction of the module $M$ to  ${\mathcal X}(0)$. \end{theorem}

One of application of  the Multiplicity Polar theorem we will need pertains to the intersection multiplicity of two modules as defined by Serre (\cite{S}).  Given modules $M_1\subset F_1$ and $M_2\subset F_2$, $F_i$ free $ {\mathcal{O}}_{X^{d},x}$ modules  of rank $p_i$ as above, Serre's intersection number is the alternating sum of the lengths of the $Tor^i(F^{p_1}/M_1, F^{p_2}/M_2)$. (Cf. \cite  {G-G} for all of the necessary hypotheses for this number to be defined.)  If the local ring of $X$ is Cohen Macaulay, then we can hope to calculate this number as a length.  The two candidates are $e(M_1, {\mathcal{O}}_{C(M_2),x})$ and $e(M_2,  {\mathcal{O}}_{C(M_1),x})$. If the local ring of $X$ is Cohen Macaulay, then these multiplicities are the colength of the ideal of maximal minors of each module (\cite {G-G} cor 2.4). In \cite {G-G} Theorem 2.3 and Corollary 2.5 it is shown that these numbers are equal and both are equal to  Serre's intersection number. 

We now discuss the framework for addressing equisingularity problems we will use in this paper. 

We are interested in equisingularity conditions which are equivalent to the inclusion of the partial derivatives of a map germ with respect to the parameter values in the integral closure of a module. The conditions which can be studied in this way include the Whitney conditions, and Thom's A$_f$ condition. 

Given the total space of a  family of spaces and the  module, the inclusion conditions depend on the polar varieties of the module.

Suppose we have a family of $d$-dimensional spaces parametrized by a smooth $Y^k$; then in \cite{Gaff1} it is shown that if the codimension $d$ polar variety of $M$ is empty, and $h$ is generically in the integral closure of $M$, then it is in the integral closure of $M$. The multiplicity polar theorem shows that to ensure that the polar variety of codimension $d$ is empty, i.e. has multiplicity $0$ over $Y$, it suffices that $e(M(0),N(0))+mult_{y}\Gamma_{d}(N)$ is the same as the multiplicity of $e(M(y),N(y)$ for a generic value of $Y$. 

In this setting it then becomes important to see what the correct choice of $N$ is, and to control its codimension $d$ polar variety. 

Sometimes the correct choice of $N$ is clear from the construction of $X$. For example, if $X$ is a section of a stable singularity $S$, by a map $M(X)$, and we vary $X$ through sections of 
$S$, using $M(\cX)$, then the correct choice of $N(X)$, is just the restriction to $X$ of the infinitesimal deformations of $S$. (Think of fixing the map to the ambient space of $S$, then varying $S$ to induce deformations of $X$.) Since $S$ is stable, the sheaf of infinitesimal deformations of $S$ is just the sheaf associated with the Jacobian module of $S$. If $\cX$ is a family of sections, and $i$ the inclusion of $X$ in family then it is obvious that $M(X)^*(JM(S))=N(X)=i^*(N(\cX))=i^*(M(\cX)^*(JM(S)))$.

If $S\subset \mathbb C^p$ is stable, then we have a nice description of the polar varieties of $N(X)$. 

\begin{Prop} Suppose $S$ is a stable singularity, and $M\:\mathbb C^q\to \mathbb C^p$ defines $M^{-1}(S)=X^d$ with expected codimension. Then $$\Gamma_i(N(X))=M^{-1}(\Gamma_i(C(S)))$$ for $i<d$.\end{Prop}
\begin{proof} Since $S$ is stable, $\Gamma_i(N(S))=\Gamma_i(C(S))$. Since $X$ has the expected dimension, sets such as as $M^{-1}(\Gamma_i(C(S)))$ have expected codimension as well for $i\le d$. This implies that $JM(S)$ and $M^*(JM(S))$ have the same generic rank. Now by Proposition 2.2, we know that  $\Gamma_i(C(S))$ and $\Gamma_i(M^*N)$ are defined by the points where $g+i-1$ generic generators of the sheaves have rank less than $g$, where $g$ is the common generic rank. Now we can choose $g+i-1$ generators of $JM(S)$ which are generic for $JM(S)$ and which pullback to generic generators for $M^*(N)$. The pullbacks have rank less than $g$ exactly on
$M^{-1}(\Gamma_i(C(S)))$, which finishes the proof.
\end{proof}

In what follows we will want to consider the intersection number of a map $M:\mathbb C^q\to \mathbb C^p$ with subsets $C$ of $S$. Recall how to think of this. Consider the graph of $M$ in $\mathbb C^q\times \mathbb C^p$, and intersect with $\mathbb C^q\times C$. If the intersection is proper  and of dimension $0$ then the intersection number is well defined, and if we move the graph so that it intersects $\mathbb C^q\times C$ transversely and at smooth points, the number of points will be the intersection multiplicity. 

In the construction of our invariants which depend only on $X$ and its landscape, we will be particularly interested in maps $\tilde M$, which define a deformation of $X$ to its smoothing in the fixed landscape.

 \begin{theorem} Suppose $S$ is a stable singularity, and $M:\mathbb C^q\to \mathbb C^p$ defines $M^{-1}(S)=X^d_M$ with expected codimension, $M$ is proper. Suppose $\tilde M$ is the map which defines a deformation $\tilde X_{\tilde M}$ of $X_M$ to a smoothing. Then $$mult_{{\mathbb C}}\Gamma_d(N(\tilde X_{\tilde M}))=M({\mathbb C}^q)\cdot \Gamma_d(S)$$.  

\end{theorem}

\begin{proof} We know from the previous proposition that $mult_{{\mathbb C}}\Gamma_d(N(\tilde X_{\tilde M}))=mult_{{\mathbb C}}\tilde M^{-1}(\Gamma_d(S))$. Given $M$ we can choose our smoothing so that $\tilde M$ is transverse to both $S$ and $\Gamma_d(S)$. Then $mult_{{\mathbb C}}\tilde M^{-1}(\Gamma_d(S))=M({\mathbb C}^q)\cdot \Gamma_d(S)$.
\end{proof}
This result shows that  $mult_{{\mathbb C}}\Gamma_d(N(\tilde X_{\tilde M}))$ is independent of the choice of family. In the following sections we will compute this intersection number for maximal rank determinantal singularities, essentially by calculating the polar varieties of $\Sigma_1$.

Since the polar varieties of $S$ are intimately connected with our invariant, and the polar varieties are obtained by intersecting the conormal $C(S)$ of $S$ with enough generic hyperplanes, then projecting to $S$, we next prove a result about $C(\overline{\Sigma}_r)$, the matrices of kernel rank $r$, as $\overline{\Sigma}_r$ is a good first example of an $S$.

We know that the fiber to the normal bundle to the smooth manifold $\Sigma_r$ at $M\in \Sigma_r$, is $Hom(K(M),C(M))$ where $K(M)$ denotes the kernel of $M$ and $C(M)$ denotes the cokernel, which we think of as the vectors in ${\mathbb C}^{n+k}$ which annihilate the image of $M$. So we do not treat this case in the next proposition.

So up to some identifications, the fiber of $C(\overline{\Sigma}_r)$ at $M$ is inside the projectivization of $Hom(K(M),C(M))$, which we denote by ${\mathbb P}Hom(K(M),C(M))$. Let $\Sigma_r(M)$ denote the elements of $Hom(K(M),C(M))$ of kernel rank $r$.

We are guided by the following result of Gelfand, Kapranov and Zelevinsky about dual varieties. 
They considered the projective tangent cone of $\overline{\Sigma}_r$ at the origin in $Hom ({\mathbb C}^n, {\mathbb C}^{n+k})$, and computed its dual variety. Let $X_r$ denote the projective variety determined by $\overline{\Sigma}_r$. Then, they showed that the dual cone of $X_r$ is $X_{n-r}$  (\cite{GKZ}, p36, Prop. 4.11). 

If $M\in Hom(\mathbb C^n,\mathbb C^{n+k})$, then we denote $\mathbb P(\overline{\Sigma}_r(M))$ by $X_r(M)$.

\begin{Prop} Suppose $M$ is in  $\Sigma_s$, $s> r$. Then the fiber of the conormal of $C(\overline{\Sigma}_r)$ at $M$ is $X_{s-r}(M)$.

\end{Prop}
\begin{proof} 
    The fiber of $C(\overline{\Sigma}_r)$ at $x$ consists of the union of the dual cone of the tangent cone at $x$ and the dual cones of certain subcones which are exceptional  (\cite {Le-T}). In our proof we will compute the dual cones of the tangent cones, then show that there are no exceptional subcones. 

%Since $\overline{\Sigma}_r$ is stable, its Jacobian module and normal module are the same, so $C(\overline{\Sigma}_r)$ and $N(\overline{\Sigma}_r)$ have the same equations. 

We will see in the next section that some of the equations of $C(\overline{\Sigma}_r)$ will be 
$$[T_{i,j}]^tM=0$$  $$M[T_{i,j}]^t=0$$
where $[T_{i,j}]$ is a matrix of indeterminates whose entries are coordinates on ${\mathbb P}^{n(n+k)-1}$ and $M\in Hom(\mathbb C^n,\mathbb C^{n+k})$. As we calculate the fiber of $C(\overline{\Sigma}_r)$ at $M$, we fix $M$.

These equations can be re-written by taking transpose as:

$$M^t[T_{i,j}]=0$$
$$[T_{i,j}]M^t=0$$

The first equation implies that the image of $[T_{i,j}]$ is in $K(M^t)$; in turn this implies that the image of $[T_{i,j}]$ is in $C(M)$. The second equation implies $K(T)\supset im(M^t)$. 
Notice that we have that $\mathbb C^n=K(M)+ im(M^t)$. So, if $L\in Hom(K(M),C(M))$, then we can view $L$ as in $Hom(\mathbb C^n,\mathbb C^{n+k})$, using the inclusion of $C(M)$ into 
$\mathbb C^{n+k}$, and extending by $0$ over $im(M^t)$. Extending in this way, $L$ will satisfy both sets of equations. Further if $L\in X_{s-r}(M)$, then $L\in X_{n-r}$, because we have added $im(M^t)$ to the kernel.

We assume coordinates chosen so that the upper left corner of $M$ is an $n-s\times n-s$ identity matrix, all other entries $0$. Let $H$ be the set of matrices whose upper left corner is an $n-s\times n-s$ identity matrix, and lower right corner is a $(s+k)\times s$ matrix of indeterminates. The family $H$ is a transverse section of $\Sigma_s$ at $M$ of complementary dimension. By  \cite{GKZ}, p36, Prop. 4.1, we know that the dual to the tangent cone to  $\overline{\Sigma}_r\cap H$ is $X_{s-r}(M)$. The elements of this last set give elements of 
$X_{n-r}$ as we remarked earlier. 

The tangent vectors to $Hom(\mathbb C^{n},\mathbb C^{n+k})$ can be thought of as elements of $Hom(\mathbb C^{n},\mathbb C^{n+k})$ as well. To check if a tangent vector is in the hyperplane defined by a linear form $L$ (which also can be thought of as an element of $Hom(\mathbb C^{n},\mathbb C^{n+k}))$, take the entry-wise dot product of $L$ and $B$.

Going back to $M$, the normal space consists of the matrices of whose lower right corner is a $(s+k)\times s$ matrix, all other entries $0$. These other entries give the tangent vectors to 
$\Sigma_s$ at $M$. This implies that the extensions of  $X_{s-r}(M)$ to $X_{n-r}$ are the duals to the tangent cone to $\Sigma_s$ at $M$, as the analytic triviality of $\overline{\Sigma}_r$ along $\Sigma_s$ implies the tangent cone of $\overline{\Sigma}_r$ is a product. 

Now we have to show that there are no exceptional subcones. From the group action of $GL(n+k)\times GL(n)$ on the  matrices  we know that the germ of ${\Sigma}_r$ at every smooth point of $\Sigma_r$ is analytically equivalent, so the exceptional cones would have to be  the tangent cones to the singular set of $\overline{\Sigma}_r$ at $M$. Suppose one of these cones, say that of $\overline{\Sigma}_s$, is an exceptional cone. Then every tangent hyperplane to  ${\Sigma}_s $ must be a limit of tangent hyperplanes from $\Sigma_r$. Suppose $M\in {\Sigma}_s$. Consider sequences of points from $C(\overline{\Sigma}_r)$, which converge to $M,H$, $H$ a hyperplane in $\mathbb P(Hom(K(M),C(M)))$ which is of rank $s$. This implies that extension of $H$ to $Hom(\mathbb C^{n},\mathbb C^{n+k})$  has rank $s$ and kernel $imM^t$. The dimension of $imM^t$ is $n-s$. Meanwhile the limiting hyperplanes from $\Sigma_r$ have kernel dimension $n-r>n-s$, which is impossible.
\end{proof} 

\begin{cor} $\Gamma^u(\overline{\Sigma}_r)$ is empty for $u\le (n-r)(n+k-r)-1$.
\end{cor}
\begin{proof} Let $h$ denote the dimension of $Hom(\mathbb C^{n},\mathbb C^{n+k})$. Let $c$ denote the codimension of $\overline{\Sigma}_{n-r}$ in $Hom(\mathbb C^{n},\mathbb C^{n+k})$. By the last proposition $c$ is also the codimension of the fiber of $C(\overline{\Sigma}_r)$ over the origin in $\mathbb PHom(\mathbb C^{n},\mathbb C^{n+k})$.

If $j$ satisfies 
$$j\ge h-c,$$
then for a generic choice of $j$ hyperplane sections, the polar determined by these sections will be empty as the number of sections is greater than the dimension of the fiber over zero.

Using $j$ sections determines the polar of dimension $h-1-j$, so the polar of dimension less than or equal to $h-1-(h-c)=c-1=(n-r)(n+k-r)-1$ is empty.
\end{proof}
For example, this implies that for maximal rank deteminantal singularities on ${\mathbb C}^5$ with presentation matrix $M$ of size $(3,2)$, $mult_{{\mathbb C}}\Gamma_d(N(\tilde X_{\tilde M}))$ is always $0$.

Now we describe how to use the information coming from a smoothing of $X$ to study the polar varieties of a family containing $X$ in the fixed landscape. We do this in a little more generality than we need for our applications.

Here is our setup. Let $\mathcal{X}_{Y} \subset \mathcal{X}_Z$ be the total spaces of two families over smooth bases $(Y,0)$ and $(Z,0)$ such that $(Y,0)$ is a proper closed subspace of $(Z,0)$, and $Z \subset \mathcal{X}_Z$. Assume that the central fiber $\mathcal{X}_{0}$ of the two families is $(X,0)$. Let $M$ be a sheaf of modules on $\mathcal{X}_Z$, which is a subsheaf of a free sheaf $F$, of rank $g$ off a closed subspace $C(M)$ finite over $(Z,0)$.

Let $Z_0$ be a small enough neighborhood of $Z$ around $0$. Let $y \in Z_0 \cap Y$ be a generic point of $Y$ to be specified later, and let $Z_y$  be a small enough neighborhood of $y$ such that $Z_y \subset Z_0$.

Work over $(\mathcal{X}_Z,0)$. Let $\pi_{M}$ be the structure map $\pi_{M}: \mathrm{Proj}(\mathcal{R}(M)) \rightarrow (\mathcal{X}_Z,0)$.  Consider the composition of maps

$$\pi_{M}^{-1} (C(M)) \hookrightarrow  \mathcal{X}_Z \times \mathbb{P}^{g(M)-1} \xrightarrow{pr_2} \mathbb{P}^{g(M)-1}.$$
As $M$ is of generic rank $g$, by the Kleiman Transversality Theorem (see \cite{K}), the intersection of $\pi_{M}^{-1} (C(M))$ with a general plane $H_{d+g-1}$ from  $\mathbb{P}^{g(M)-1}$ of codimension $d+g-1$, is of dimension at most $\dim Z - 1$. Therefore, for a generic $z \in Z_0$ the fiber over $z_0$ of the projection $\Gamma_d(M)$ of $\mathrm{Proj}(\mathcal{R}(M)) \cap H_{d+g-1}$ to $\mathcal{X}_Z$ consists of the same number of points, each of them appearing with multiplicity one, and at which the rank of $M$ is maximal. Denote this number by $\mathrm{mult}_{Z_0}\Gamma_d(M)$.

In the same way define $\mathrm{mult}_{Y}\Gamma_d(M)$ for $(\mathcal{X}_Y,0)$, where we identify $M$  with its image in $F \otimes_{\mathcal{O}_{\mathcal{X}_Z}}\mathcal{O}_{\mathcal{X}_Y}$. Let $y \in Z_0$ be generic enough so that the cover $\Gamma_d(M) \rightarrow (Y,0)$ is unramified at $y$. Denote by $y_1,\ldots,y_k$ the points in the fiber $C(M)_y$. Define $\mathrm{mult}_{Z_{y},y_i}\Gamma_d(M)$ as above for small enough neighborhood $(\mathcal{X}_y,y_i)$. Set
$$\mathrm{mult}_{Z_y}\Gamma_d(M):=\sum_{i=1}^{k}\mathrm{mult}_{Z_{y},y_i}\Gamma_d(M).$$
The following result shows that the invariants $\mathrm{mult}_{Z_0}\Gamma_d(M)$ and $\mathrm{mult}_{Z_y}\Gamma_d(M)$
control the presence of $\Gamma_d(M)$, which in our applications is the obstruction to various equisingularity conditions for a suitable choice of $M$.

\begin{theorem}\label{covering argument}
We have
\begin{equation}\label{key polar eq}
\mathrm{mult}_{Z_0}\Gamma_d(M) - \mathrm{mult}_{Z_y}\Gamma_d(M)=\mathrm{mult}_{Y}\Gamma_d(M).
\end{equation}
\end{theorem}
\begin{proof} Let $H_{d+g-1}$ be a general plane that produces the covers $\Gamma_d(M) \rightarrow (Z,0)$ and $\Gamma_d(M) \rightarrow (Y,0)$. We claim that $H_{d+g-1}$ gives each of the covers $\Gamma_d(M) \rightarrow (Z_y,y)$, where by abuse of notation $M$ is identified with $M \otimes_{{\mathcal{O}_{\mathcal{X}_Z}}}\mathcal{O}_{\mathcal{X}_{y},y_i}$. Indeed, for each $i$, let $(\mathcal{X}_y,y_i)$ be a small enough neighborhood so that it avoids the branches of the cover $\Gamma_d(M) \rightarrow (Z,0)$ that do not pass through $y_i$. As the codimension of $H_{d+g-1}$ is right and the cover $\Gamma_d(M) \rightarrow (Z,0)$ is unramified for generic $z \in Z_y$, we get that the branches passing through $y_i$ form the polars $\Gamma_d(M) \rightarrow (Z_y,y)$ for the various restrictions of $M$ to $(\mathcal{X}_y,y_i)$. Observe that for generic $z \in Z_y$ close enough to $y$ the degree of the cover $\Gamma_d(M) \rightarrow (Z,0)$ is $\mathrm{mult}_{Z_0}\Gamma_d(M)$. Over $y$ some of these branches merge at the $y_i$ and their number is $\mathrm{mult}_{Z_y}\Gamma_d(M)$ as showed above. The rest of the branches intersect $\mathcal{X}_y$ at points where the rank of $M$ is maximal by our choice of $y$. Moreover, their number is $\mathrm{mult}_{Y}\Gamma_d(M)$ by our choice of $H_{d+g-1}$. This finishes the proof.
\end{proof}
Throughout the rest of the paper we will apply Theorem \ref{covering argument} in the following setting: $(X,0)$ is an isolated singularity, $Z$ is a smoothing component of the miniversal base space of $(X,0)$, and $M$ is an appropriate modification of the relative Jacobian module of the induced deformation $\mathcal{X}_Z$. Because $Z$ is a smoothing component, we will compute $\mathrm{mult}_{Z_0}\Gamma_d(M)$ as $\mathrm{mult}_{S}\Gamma_d(M)$ where $S \subset Z$ is a smooth curve and the generic fiber of the induced deformation $\mathcal{X}_{S}$ of $(X,0)$ is smooth. By Theorem \ref{covering argument} we will know that if $\mathrm{mult}_{Z_y}\Gamma_d(M)$ is independent of $y$, then $\Gamma_d(M)$ will be empty. In turn, this will be used to ensure that the integral closure conditions on which an equisingularity condition depends, hold.

\section{The determinantal normal module}

In this section we study the determinantal normal module, which is the first order infinitesimal deformations for our choice of landscape. We first set the stage for applying the framework from \cite{Gaff1} to families of determinantal singularities. 

In \cite{Gaff1}, in applications to families of $d$-dimensional isolated singularities, given a module $M$, equisingularity conditions were controlled using invariants of the pair  $(M,N)$ where $N=H_{d-1} (M)$. Here $H_{d-1} (M)$  is the module of elements which were in the integral closure of $M$ off a subset of codimension at least $d$. If $M$ is the Jacobian module $JM(X)$ or $mJM(X)$, then there is a link between $H_{d-1} (M)$ and the infinitesimal deformations of $X$.

Let $I$ be the ideal of $\cO_n$ defining $X^d$; denote the set of all elements  $h\in I$  such that the
partial derivatives of $h$ are in $I$ by
$\int I$. Note that we can identify $\cO^p_X$ with its dual $\hom(\cO^p_X, \cO_X)$. If $I$ has $p$ generators, we have the following short exact sequence of
$\cO_X$ modules.

$$0\To R\To \cO^p_X\To I/I^2\To 0$$ 

Here $R$ is the module of relations. Denote the map to $I/I^2$ by $j$. This gives the injection 
$$0\to \hom(I/I^2,\cO_X)\to \hom(\cO^p_X, \cO_X).$$

So, we can identify elements in the image of this last inclusion with their preimages. Note that each partial derivative operator defines an element of 
$\hom(I/I^2,\cO_X)$. Denote the submodule of $\hom(I/I^2,\cO_X)$ generated by the elements defined by the partial derivative operators by $D$. Then 

\begin{Prop}  With the identification of $\cO^p_X$ with $\hom(\cO^p_X, \cO_X)$, the module $JM(X)$ is the image of $D$ under the inclusion of $\hom(I/I^2,\cO_X)$ in
$\hom(\cO^p_X, \cO_X)$, and $H_0(JM(X))$ is the image of $\hom(I/\int I,\cO_X)$. \end {Prop}

\begin{proof} Cf. \cite{Gaff-AG} Proposition 5.1.

If we assume $I$ is radical then in fact we have the module $H_0(JM(X))$ is the image of the normal module,  $\hom(I/ I^2,\cO_X)$ for then the inclusion  of $\hom(I/\int I,\cO_X)$ into  $\hom(I/I^2,\cO_X)$ is an isomorphism. (Cf. \cite {P} lemma 1.15.)
\end{proof}
It is not hard to see, that if the singular locus of $X^{d+k}$ has codimension $d$, then $H_0(JM(X))=H_{d-1}(JM(X))=H_{d-1}(mJM(X))$.

If $X,0$ is the germ of a  Cohen-Macauley (CM)  subvariety of $\mathbb{C}^q,0 $ of codimension 2, then its ideal is generated by the maximal minors of an $(n+1)\times n$ matrix with entries in $\cO_q$. Let $M_X$ denote the presentation matrix of $X$, $M_{i,j}$, the $(i,j)$ entry of $M_X$.  We have a description of the normal module in terms of the presentation matrix.
\begin{Prop} Suppose $X,0$ is the germ of a  Cohen-Macauley (CM)  subvariety of $\mathbb{C}^q,0 $ of codimension 2, then normal module of $X$ is given by
$$ Mat(n+1,n;\cO_q) /  Im(g)$$
where $g$ is the map 
$$Mat(n+1,n+1;\cO_q)\oplus \Mat(n,n;\cO_q) \overset{g}{\rightarrow}  \Mat(n+1,n;\cO_q)$$

mapping $(A,B)  \mapsto  AM+MB$.
\end {Prop}
\begin{proof}\cite{FK1} lemma 2.6.
\end{proof}

If $X$ is determinantal, but not necessarily of codimension 2, then by considering only deformations of $X$ which arise as deformations of the presentation matrix, $M_X$, we can define $N_D(X)$, the determinantal normal module as the first order infinitesimal deformations of $X$ coming from deformations of the presentation matrix. If $X^d$ has an isolated singularity, since $N_D(X)$ sits between $JM(X)$ and $H_d(JM(X))$, it is clear that $JM(X)$ has finite colength inside $N_D(X)$, hence $e(JM(X), N_D(X))$ is well defined. As the previous proposition shows, the notions of determinantal normal module and normal module coincide for the codimension 2 case.

The landscape of determinantal singularities consists of determinantal singularities, with allowable deformations those which arise by deforming the presentation matrix of the singularity, with $N_D(X)$ as the first order infinitesimal deformations of $X$.

It is not hard to describe a set of generators for  $N_D(X)$. If the size of $M_X$ is $(n+k)\times n$, let $\Delta(l_1,\ldots,l_k)$ denote the maximal minor of $M_X$ obtained by deleting rows $(l_1,\dots,l_k)$. (We assume 
$l_i<l_{i+1}$ for $1\le i\le k-1$.) Let $\delta_{i,j}$ be the $(n+k)\times n$ matrix with $1$ in the $(i,j)$ entry, other entries $0$. Consider the deformation of $M_X$ given by $M_X+t\delta_{i,j}$. Each such deformation
gives a column in a matrix of generators for $N_D(X)$, so there are  $(n+k)n$ columns in a matrix of generators for $N_D(X)$. The $(l_1,\dots,l_k)$ entry of the $(i,j)$ column is gotten by taking the $(l_1,\dots,l_k)$ minors of  $M_X+t\delta_{i,j}$, and taking the linear part of this in $t$. This entry is denoted by $m_{i,j}(l_1,\dots,l_k)$ and is the cofactor of $M_{i,j}$ in the expansion of $\Delta(l_1,\dots,l_k)$. Of course it is $0$ if $i=l_k$ for some $k$. 

If we have a family $\cX$ of determinantal singularities defined by deforming the entries of a presentation matrix, then it is clear that $N_D(\cX)$ specializes to members of the family, and to sub-deformations. This set of generators also shows that $N_D$ is universal.

\begin{Prop} In the landscape of determinantal singularities ${\overline\Sigma}_r\subset Hom({\mathbb C}^n,{\mathbb C}^{n+k})$ is a stable singularity.
\end{Prop}
\begin{proof} The minors of size $n-r+1$ are a set of generators of the ideal defining  ${\overline\Sigma}_r$. The generators of the Jacobian module they give are computed as follows: take the matrix $M_{ID} +t\delta_{i,j}$, take the column vector of minors of the desired size and take derivative with respect to $t$. This is the same as as taking linear parts in $t$. Hence $JM({\overline \Sigma}_r)=N_D({\overline \Sigma}_r)$.
\end{proof}

\begin{cor}  Suppose $S={\overline\Sigma}_r$,   $M:\mathbb C^q\to Hom({\mathbb C}^n,{\mathbb C}^{n+k})$ defines $M^{-1}(S)=X^d_M$ with expected codimension, $X^d_M$ smoothable. Suppose $\tilde M$ is the map which defines a deformation $\tilde X_{\tilde M}$ of $X_M$ to a smoothing. Then $$mult_{{\mathbb C}}\Gamma_d(N(\tilde X_{\tilde M}))=M({\mathbb C}^q)\cdot \Gamma_d({\overline\Sigma}_r).$$
\end{cor}
\begin{proof} From the previous proposition,  ${\overline\Sigma}_r$ is stable and universal, hence the corollary follows by Theorem 3.5.
\end{proof}

Following Fr\"{u}bis-Kr\"uger, given $X\subset {\mathbb C}^q$ with presentation matrix $M_X$, we can consider the group action $G$ on $M_X$ given by multiplication of $M_X$ with invertible matrices on the left and right, and compositions with coordinate changes on $ {\mathbb C}^q$. The action of this group is well understood, and its extended tangent space in the sense of Mather is the quotient module of $N_D(X)$  by  $JM(X)$, which has finite colength if $X$ has an isolated singularity. The theory of these actions then shows that $M_X$ is finitely determined in the sense that perturbing the entries of $M_X$ by a sufficiently high power of the maximal ideal doesn't change $X$ up to a coordinate change of $ {\mathbb C}^q$. Further, $M_X$ has a deformation $M_{\cX}$ with smooth finite dimensional base $Y$ such that any deformation of $X$ defined by deforming the entries of $M_X$ can be induced from $M_{\cX}$.  It follows that any two smoothings of $X$ defined by deforming the entries of $M_X$ are homeomorphic.

In the next section we want to calculate the multiplicity of the polar of $N_D(\cX)$  in the case where the generic fiber of the deformation is smooth, so we want to describe $\Projan(\cR(N_D(X)))$. The description we give will apply equally to $\Projan(\cR(N_D(\cX)))$.

Since the generators of $N_D(X)$  are in one to one correspondence with the entries of $M_X$, we know that

$$\Projan(\cR(N_D(X)))\cong \cO_X[T_{i,j}]/I, \hskip 2em 1\le i \le n+k, 1\le j\le n,$$where $I$ is the ideal of relations between the $T_{i,j}$ under the map which sends $T_{i,j}$ to the ${i,j}$ generator of $\cR(N_D(X))$.

\begin{Lemme} $I$ contains the entries of the matrices $[T_{i,j}]^tM_X$ and $M_X[T_{i,j}]^t$.
\end{Lemme}
\begin{proof} We can view the map from $Mat(n+1,n;\cO_q)$ to the normal module in the proof of lemma 2.6  of \cite{FK1}, as a map from  $ \cO_X[T_{i,j}]$ to $\cR(N)$. So we can use the image of $g$ to find relations in this case by translating to $ \cO_X[T_{i,j}]$. The analogous map from $ \Mat(n+k,n;\cO_q)$ to $N_D(X)$, still carries the image of $g$ to the trivial deformations, hence to zero in $N_D(X)$. So the image of $g$ still gives elements of $I$. We claim these elements are entries of the matrices $[T_{i,j}]^tM_X$ and $M_X[T_{i,j}]^t$. We trace through the relations between these two settings to show this. Suppose $\delta_{i,j}$ is the matrix in $\Mat(n+k,n+k;\cO_q)$ with $1$ in the $(i,j)$ entry and zero elsewhere. Then the matrix $\delta_{i,j}M_X$ has the $i$-th  row as the only non-zero row  with entries $(m_{j,r})$. This gives the element $\sum_{r=1}^n m_{j,r}T_{i,r} $ in $I$ in $ \cO_X[T_{i,j}]$. In turn this is $(M_X[T_{i,j}]^t)_{j,i}$. The computation for $[T_{i,j}]^tM_X$ is similar.
\end{proof}

\begin{rem} The lemma means that the entries of $[T_{i,j}]^tM_X=0$ and $M_X[T_{i,j}]^t=0$ are some of the equations of $\Projan(\cR(N_D(X)))$. If we work at a point $x$ where $X$ is smooth, hence $M_X$ has rank $n-1$, then the values of the rows and columns of $[T_{i,j}]$ must all be in the kernel of $M_X(x)$ and $M_X^t(x)$ respectively. This implies that the entries of $[T_{i,j}]$ give a matrix of rank $1$.

\end{rem}

We will use the remark in the proof of the next theorem which will allow us to decompose the computation of $mult_{y}\Gamma_{d}(N_D(X))$ into manageable pieces.

There are two interesting transforms we can make of $X$ using $M_X$. Let $X_{n-1}$ denote the set of points  of $X$ where $M_X$ has rank $n-1$; we define:

$$X_M:=\overline{\{(x,l_1,l_2)| x\in X_{n-1}, l_1\in {\mathbb P}(ker(M^t_X(x))), l_2\in {\mathbb P}(ker(M_X(x))\}}$$
$$X_T:=\overline{\{(x,l)| x\in X_{n-1}, l\in {\mathbb P}(ker(M_X(x))\}}$$ 
 where ${\mathbb P}(ker(M_X(x))$ is the projectivization of the kernel of $M_X(x)$. Hence $X_M$ is contained in $X\times {\mathbb P}^{n+k-1}\times  {\mathbb P}^{n-1}$, while 
$X_T\subset X\times {\mathbb P}^{n-1}$. The transform $X_T$ is known as the Tjurina transform.

There is a third transform of the ambient space ${\mathbb C}^{d+k+1}$, 
$${\mathbb C}^{d+k+1}_{T^t}:= \overline{\{(x,l)| x\in{\mathbb C}^{d+k+1}, l\in {\mathbb P}(ker(M^t_X(x))\}}.$$
Hence ${\mathbb C}^{d+k+1}_{T^t}$ is contained in ${\mathbb C}^{d+k+1}\times {\mathbb P}^{n+k-1}.$

Denote the $(n+k,n)$ matrix with entries $a_{i,j}=x_{i,j}$, where  $x_{i,j}$ are coordinates on ${\mathbb C}^{n(n+k)}$ by $ID$. Then $ID$  is a biholomorphic map from ${\mathbb C}^{n(n+k)}$ to $Hom({\mathbb C}^n,{\mathbb C^{n+k}})$, so the above constructions also apply to $\overline{\Sigma}_1$. We denote $X_{ID}$ by $\tilde{\overline{\Sigma}}_1$.

We will alter the presentation matrix $M$ by dropping rows; this will induce new transforms of the three types defined above.

\begin{theorem} Suppose $X$ is a maximal rank reduced determinantal singularity, with $X_{n-1}$ dense in every component of $X$, then $\Projan \cR(N_D(X))$ is isomorphic to $X_M$ as sets.
\end{theorem}

\begin{proof} Since both sets are defined by the closures of the points over $X_{n-1}$ it suffices to work on this set. Let $(S_1,\ldots , S_{n+k})$ be coordinates on ${\mathbb P}^{n+k-1}$ and $(T_1,\dots, T_{n})$ be coordinates on ${\mathbb P}^{n-1}$. Consider the Veronese embedding ${\mathbb P}^{n+k-1}\times {\mathbb P}^{n-1}$ in ${\mathbb P}^{n(n+k)-1}$ given by $(S_1,\dots, S_{n+k})\times (T_1,\dots, T_{n})$ maps to $T_{i,j}=S_i T_j$. This embedding sends the points of $X_M$ over  $x\in X_{n-1}$ to matrices of rank $1$ whose rows are multiples of a fixed non-zero kernel vector of $M_X(x)$ (hence the matrix has rank 1) and whose columns are in the kernel of $M(x)^t$. The set of such matrices  is a subvariety of dimension $(n-1)+(n+k-1)$. The fiber of $X_M$ over $x$ by the remark maps to points containing the fiber of $\Projan \cR(N_D(X))$. Further the fiber dimension of $X_M$ is clearly $k$ as the kernel of $M^t(x)$ has dimension $k+1$. Meanwhile the fiber dimension of $\Projan \cR(N_D(X))$ is one less than the rank of the Jacobian module of $X$, which is the expected codimension of $X$, so the fiber dimension is $k$ also. Since the image of the fiber of $X_M$ over  $x$ is irreducible as is the fiber of  $\Projan \cR(N_D(X))$, they are the same. Hence the closure of the image of $X_M$ is the same as $\Projan \cR(N_D(X))$.\end{proof}

\begin{cor} $C(\overline{\Sigma}_1)= \Projan  \cR(N_D(\overline{\Sigma}_1))=\tilde{\overline{\Sigma}}_1$
\end{cor}

\begin{proof} Since $\overline{\Sigma}_1$ is stable, $C(\overline{\Sigma}_1)= \Projan  \cR(N_D(\overline{\Sigma}_1))$ and the third equality follows from the Theorem.
\end{proof}

In the next section we will use this theorem to compute the degree over the base $Y^1$ of the polar variety of dimension $1$ of $N_D(\cX)$, where $\cX$ is the total space of the deformation, and a generic fiber is smooth. As we have seen $\Projan \cR(N_D(\cX))$ is a subset of $\cX\times {\mathbb P}^{n(n+k)-1}$. The hyperplane class on this space denoted $h$ is represented by $\cX\times H$, where $H$ is a hyperplane in $ {\mathbb P}^{n(n+k)-1}$. The hyperplane classes on $\cX\times {\mathbb P}^{n-1}$ and $\cX\times {\mathbb P}^{n+k-1}$ are denoted by $h_2$ and $h_1$ respectively. As classes, the pullback of $h$ to $\cX\times{\mathbb P}^{n+k-1}\times {\mathbb P}^{n-1}$ by the Veronese $V$ is $h_1+h_2$. Denote the fiber over the origin in $\cX$ of $\Projan \cR(N_D(\cX))$ by $E$. If $N_D(\cX)$ was an ideal, this would be the fiber of the exceptional divisor of the blow-up of $\cX$ by $N_D(\cX)$.

\begin{theorem} Suppose $\cX^{d+1}$ is a maximum rank determinantal singularity which is a smoothing of a maximal rank determinantal singularity $\cX^d$, with smooth base $Y^1={\mathbb C}^1$. The degree of the polar variety of $\Projan \cR(N_D(\cX))$ over $Y$ at the origin, $\Gamma_d( N_D(\cX))$ is $(h_1+h_2)^{d+k}\cdot E$.
\end{theorem}

\begin{proof} By definition the degree of $\Gamma_d( N_D(\cX))$  over $Y$ at the origin is the degree of the projection to $Y$ at the origin of $\Gamma_d( N_D(\cX))$. In turn, $\Gamma_d( N_D(\cX))=\pi_{\cX}(\Projan \cR(N_D(\cX))\cap h^{d+k})$. We can assume the representative of $h^{d+k}$ chosen so that it is transverse to all components of $E$ of maximal dimension which is $d+cod(X)-1=d+k$. Each point of intersection contributes $1$ to the degree, so the degree is $V^*(h^{d+k})\cdot V^*E=(h_1+h_2)^{d+k}\cdot V^* E$.
\end{proof}

Define $\Gamma_{i,j}(N_D(\cX))$ to be $\pi_{\cX}(X_M\cap h_1^ih_2^j)$. We call these the mixed polars of type $(i,j)$ of $N_D(\cX)$. Denote the degree of this mixed polar by $h_1^ih_2^j$.

\begin{cor} The degree of $\Gamma_d( N_D(\cX))$  over $Y^1$ is

$$\sum_{i=0}^{d+k}  { {d+k}\choose i} h_1^ih_2^{d+k-i}. $$

\end{cor}

\begin{proof} The reasoning is similar to the proof of theorem-the degree of the $\pi_{\cX}(X_M\cap h_1^ih_2^j)$ over $Y$ is $h^i_1h^j_2\cdot V^* E$.
\end{proof}

\begin{cor}  Suppose  $M:\mathbb C^q\to Hom({\mathbb C}^n,{\mathbb C^{n+k}}) $ defines $M^{-1}({\overline\Sigma}_1)=X^d_M$ with expected codimension, $M$ transverse to the rank stratification off $0$. Suppose $\tilde M$ is the map which defines a deformation $\tilde X_{\tilde M}$ of $X_M$ to a smoothing. Then $$\sum_{i=0}^{d+k}  { {d+k}\choose i} h_1^ih_2^{d+k-i}=M({\mathbb C}^q)\cdot \Gamma_d(S)$$.
\end{cor}
\begin{proof} This follows from the previous theorem and Theorem 2.5.
\end{proof}
The degrees of the mixed polars will be computed in the next section.

\section {Computing the Degrees of the Mixed Polars }

We define the terms we use in our formula, then prove the formulas subject to some genericity assumptions. Throughout we assume the size of $M_X$ is $(n+k)\times n$, dropping the subscript on $M_X$.

Although our goal is to calculate the multiplicity of the mixed polars associated to $N(X)$, this is done by calculating  mixed polar varieties of ${\overline\Sigma_1}$. Our calculation consist of giving polynomials in the coordinates on $Hom({\mathbb C}^n,{\mathbb C}^p)$. Typically, some of these polynomials define determinantal varieties which contain the desired mixed polar as components, the next group of polynomials define a determinantal variety some of whose  components are not in the desired mixed polar, but which do appear in the previous group and so forth. In the calculation of the multiplicity of the mixed polar this gives rise to an alternating sum of multiplicities. These sets of polynomials pull back via the presentation matrix to give mixed polars of $N({\mathcal X})$, ${\mathcal X}$ a smoothing of $X$.

 In what follows, let $M_{d+k+l}$ denote the submodule of $R^r$ where $r=(n+k)-(d+k+l)=n-d-l$, whose matrix of generators consists of the first $n-d-l$ rows of  $M$;  $R$ is the ring of either $X,0$ or $\cX,0$. (The context tells which.) In the formulas, $l$ will be 
$j-i$ where $i$ is the exponent of $h_2$ and $j$ is an index.  Because $i$ is the exponent of $h_2$, we must have $0\le i\le min \{d,n-1\}$. (This follows because the dimension of the fiber of the Tjurina transform of $\cX,0$ is $\le min \{d,n-1\}$.) Notice that $i\le d$ implies $n-d+i$, the number of rows if $j=0$, is less than or equal to the number of columns. In order for the number of rows to be non-zero for $j>0$, we need  for $n-d+i-j>0$ or $j\le n-d+i-1$.

Let $M^{c,i}_{d+k-i, j}$ denote the submodule of $R^{r_c}$ where $r_c=n+k-j-1$, whose matrix of generators consists of the last $n-i$ columns of $M$ and the first $n-d+i-j-1$ rows of $M$ as well as the last $d+k-i$ rows of $M$. Notice that if $j=0$, then $M^{c,i}_{d+k-i, j}$ has $n+k-1$ rows in all. Incrementing $j$ by $1$ drops another row. Sometimes, in the last term in our sums $j=k$; in this case $M^{c,i}_{d+k-i, j}=0$. Otherwise, we  assume $j\le i+k-1$ to ensure there are at least as many rows as columns in the matrix of generators.

The assumption in these constructions is that we are using the coordinate hyperplanes as the representatives of the powers of the $h_i$. As we shall see, dropping the last $d+k-i+j$ rows of $M$ is the effect of setting the last $d+k-i+j$ of the $S_l$ to $0$. In a similar way, dropping the first $i$ columns is the effect of setting the first $i$ of the $T_l$ to $0$. 

The genericity conditions we need are of the following type: for the presentation matrix, after row and column operations, representing the $h^r_i$ by coordinate hyperplanes, then $(h_1+h_2)^{d+k}\cdot E$ has the expected dimension and $(h_1+h_2)^{d+k+1}\cdot E$ also has the expected dimension--$0$ in the case of a one parameter smoothing. There are submatrices of $M$ that appear in our formula, and they  have the analogues of the $(h_1+h_2)^{d+k+1}\cdot E$. We also ask these intersections have the expected dimension.

To give an overview of how the computation goes, we look at one of the formulas we will prove:

$$h_1^{d+k}\cdot E= \sum_{j=0}^{min(n-d-1, k)} (-1)^j M_{d+k+j}  \cdot   M^{c,0}_{d+k, j}. $$

The terms on the right hand side are both the colength of the ideal generated by the maximal minors of the presentation matrices of the modules, and the degree over $Y$ of the projection to $Y$ of the curves defined by those maximal minors on the total space of the smoothing. If $n-d-1<0$, then the sum has no terms as $h_1^{d+k}\cdot E$ is zero.

The curves defined by each term are determinantal by genericity; with the exception of the last term of the sum, they fall into two types. The term with $j=0$ contains $\Gamma_{d+k,0}$ and  a type shared with $M_{d+k+1}  \cdot   M^{c,0}_{d+k, 1} $. In turn, the second type of curve for $M_{d+k+1}  \cdot   M^{c,0}_{d+k, 1} $ is shared with the next term in the sum. The last term in the sum can have two possible forms. If $k< n-d-1$, then the presentation matrix of $M^{c,0}_{d+k, k-1} $, which is part of the penultimate term, is a square matrix, and the curves associated with the last term ($j=k)$ are defined by the maximal minors of $M_{d+k+k}$ alone by our convention. These curves are shared with the previous term.  If $0<n-d-1\le k$, then $(n+k)-(d+k+(n-d-1))=1$ so the presentation matrix of $ M_{k+n-1}$ has $1$ row, and this term has only one type of curve, shared with the previous term. It is possible that $n-d-1$ may be negative or $0$. (For example, if $n=2$.)  In this case, $\Gamma_{d+k,0}$ is empty. The genericity conditions ensure that the sum is telescoping, and its value is the degree of $\Gamma_{d+k,0}$ over $Y$.

For the computation of the degree, it is crucial that the sets we work with are determinantal; this is checked by the next lemma.

\begin{Lemme} Suppose $M$ is an $(n+k)\times n$ matrix and the coordinates on $ {\mathbb P}^{n-1}$ and $ {\mathbb P}^{n+k-1}$ are generic in the sense of the above paragraph. Then for each pair of modules $(M_{d+k+j-i},M^{c,i}_{d+k-i, j})$ the co-supports of the modules on either ${\mathbb C}^{d+k+1}$ or ${\mathbb C}^{d+k+1}\times Y$ are determinantal. Further, the co-supports of either module in the local ring of the co-support of the other is determinantal. Hence the ideal of maximal minors of both matrices of generators defines a set of dimension $1$ finite over $Y$.\end{Lemme}

\begin{proof} The co-support of either $(M_{d+k+j-i},M^{c,i}_{d+k-i, j})$ is defined by the maximal minors of the matrix of generators. The matrix of generators of $M_{d+k+j-i}$ has size $(n-d+i-j)\times n$, so the vanishing of the ideal of minors has expected codimension $n-(n-d+i-j)+1=d-i+j+1$. Meanwhile $M^{c,i}_{d+k-i, j}$ has a matrix of generators of size $(n+k-j-1)\times (n-i)$, so the vanishing of the ideal of minors has expected codimension $n+k-j-1-(n-i)+1=k+i-j$, so the expected codimension of the intersection is $(d-i+j+1)+(k+i-j)=d+k+1$ which is the dimension of the ambient space of $X$. The genericity hypotheses ensure that the expected codimensions are realized; since the ideals defining the co-supports are ideals of minors, the result follows.\end{proof}

Given a pair of modules $(M,N)$ denote by $J(M,N)$ the ideal of maximal minors of both matrices of generators.  Next we show that the  first pair of modules we look at in each sum actually contains the desired polar varieties as components of the co-support of the pair.

\begin{Lemme} $V(J(M_{d+k-i},M^{c,i}_{d+k-i, 0}))$ contains $\Gamma_{d+k-i,i}$ as a union of components. 
\end{Lemme}
\begin{proof} The expected dimension of $\Gamma_{d+k-i,i}$ is $1$; by the previous lemma this is true for  $V(J(M_{d+k-i},M^{c,i}_{d+k-i, 0}))$ as well, so it suffices to prove that  $\Gamma_{d+k-i,i}\subset V(J(M_{d+k-i},M^{c,i}_{d+k-i, 0}))$. Suppose $x\in \Gamma_{d+k-i,i}$, assume the standard coordinates on the projective spaces are generic. Then $M(x)$, the presentation matrix with entries evaluated at $x$, satisfies two conditions.

1) $ker (M^t(x))$ contains $l\subset \{0=S_{n-d+i+1}=\dots=S_{n+k}\}$.

A basis for this space is $\{e_1,\dots, e_{n-d+i}\}$ which has codimension $d+k-i$  in ${\mathbb C}^{n+k}$. This condition implies $x\in V(J(M_{d+k-i})$.

2) $ ker (M(x))$, contains $l\subset \{0=S_1=\dots=S_i\}$.  This implies that the last $n-i$ columns of $M(x)$ are linearly dependent. In turn, this implies $x\in V(J(M^{c,i}_{d+k-i, 0}))$.\end{proof}

As we mentioned earlier, the  $\Gamma_{d+k-i,i}$ associated with a particular matrix map $M\:{\mathbb C}^q\to Hom({\mathbb C}^n,{\mathbb C}^{n+k})$, is the pull back of the mixed polar of codimension $q$ on $\bar{\Sigma_1}\subset Hom({\mathbb C}^n,{\mathbb C}^{n+k})$. We want to show that the mixed polar is independent of the choices made in the definition.

Here is some notation helpful in describing the possible choices. Let $P_{n-d+i}:{\mathbb C}^{n+k}\to {\mathbb C}^{n-d+i}$ denote the linear map which is projection on the first $n-d+i$ factors, and let $I_{n-i}:{\mathbb C}^{n-i}\to {\mathbb C}^{n}$ be the linear inclusion which maps  ${\mathbb C}^{n-i}$ into the last  $n-i$ factors of ${\mathbb C}^{n}$.  Then conditions in the last proof for elements $M$ of ${\overline\Sigma_1} \subset Hom({\mathbb C}^n,{\mathbb C}^{n+k})$ become

1') $P_{n-d+i}\circ M\in {\overline\Sigma_1}\subset Hom({\mathbb C}^n,  {\mathbb C}^{n-d+i})$.
 
 2') $M\circ I_{n-i}\in {\overline\Sigma_1}\subset Hom({\mathbb C}^{n-i},  {\mathbb C}^{n+k})$.
 
 Condition 1') is equivalent to asking that first $n-d+i$ rows of $M$  be linearly dependent, while condition 2') is equivalent to the last $n-i$ columns of $M$ are linearly dependent.
 
 If $L$ and $R$ are linear isomorphisms of ${\mathbb C}^{n+k}$ and ${\mathbb C}^n$ respectively then conditions 1') and 2') applied to 
 $P_{n-d+i}\circ L$ and $R\circ I_{n-i}$ allow us to define the mixed polars of type  $\Gamma_{d+k-i,i}$ for any desired choice of hyperplanes. 
 
 \begin{Prop} All of the mixed polars $\Gamma_{d+k-i,i}$ on ${\overline{\Sigma}_1}$ in $  Hom({\mathbb C}^n,{\mathbb C}^{n+k})$ are bi-holomorphic, hence all are equally generic.\end{Prop}
 \begin{proof} Consider the action of $Gl_{n+k}\times Gl_n$ on $  Hom({\mathbb C}^n,{\mathbb C}^{n+k})$ given by 
 $$(M,(L,R))\longrightarrow L^{-1}MR^{-1}.$$
 Each element of  $Gl_{n+k}\times Gl_n$ induces a biholomorphic map on  $  Hom({\mathbb C}^n,{\mathbb C}^{n+k})$ which preserves the rank stratification. If $M\in \Sigma_1$, and $M$ satisfies conditions 1') and 2'), then $M'=L^{-1}MR^{-1}$ satisfies the conditions 
 
 1') $P_{n-d+i}\circ L\circ M'\in {\overline\Sigma_1}\subset Hom({\mathbb C}^n,  {\mathbb C}^{n-d+i})$
 
 and 
 
 2') $M'\circ R\circ I_{n-i}\in {\overline\Sigma_1}\subset Hom({\mathbb C}^{n-i},  {\mathbb C}^{n+k})$.

 Hence the biholomorphic map induced from $(L,R)$ carries one mixed polar to the one defined using $(L,R)$.

 \end{proof}

Since these polar varieties are all equivalent by means of the group action, it follows from Kleiman's transversality theorem  (\cite{K}), that for $M$ a matrix map defining a maximal rank smoothable singularity, that the multiplicities we compute are the generic multiplicities for $M$, perhaps after applying row and column operations to $M$, as moving the image of $M$ is equivalent to fixing the image of $M$ and moving the polars.

In the next lemma we will show that $V(J(M_{d+k-i+j},M^{c,i}_{d+k-i, j}))$ contains two types of components. We now describe what these components are for this pair of modules.  Consider the presentation matrix $M$, and delete $j$ rows by deleting rows $n-d+i, n-d+i-1,\dots, n-d+i-(j-1)$ -- denote the new presentation matrix by $M(j)$. Associated to $M(j)$, we have $\Gamma_{d+k-j-i,i} (M(j))$.  The effect of dropping $j$ rows is to subtract $j$ from $k$. Notice that if $j=0$, then $M(0)=M$ and $\Gamma_{d+k-j-i,i} (M(j))=\Gamma_{d+k-i,i} (M)$.   In general we have:

\begin{Lemme} $V(J(M_{d+k-i+j},M^{c,i}_{d+k-i, j}))$ consists of two types of components, $\Gamma_{d+k-j-i,i} (M(j))$ and $\Gamma_{d+k-(j+1)-i,i} (M(j+1))$, for $j\le min( n-d+i-1, k+i)$, $0\le i\le d$, and no curve belongs to both types.  In addition if  $j=min( n-d+i-1, k+i)$, then $\Gamma_{d+k-(j+1)-i,i} (M(j+1))$ is empty. 

\end{Lemme}

\begin {proof} Suppose $j\ne min( n-d+i-1, k+i)$. By genericity we can assume that  $\Gamma_{d+k-(j+1)-i,i} (M(j))$ is empty. The proof that $\Gamma_{d+k-j-i,i} (M(j))$  is contained in $V(J(M_{d+k-i+j},M^{c,i}_{d+k-i, j}))$ is similar to the proof of the previous lemma. Consider the matrix of generators for $M_{d+k-i+j}$. This has $n-d+i-j$ rows and $n$ columns. There are two cases --either the first $n-d+i-(j+1)$ rows are linearly independent or they are not. If they are, then all of the first $n-d+i-j$ rows are in the span of these, since $x\in V(J(M_{d+k-i+j}))$.  Then $x\in V(J(M^{c,i}_{d+k-i, j}))$ implies $ x\in V(J(M^{c,i}_{d+k-i, j-1}))$, which implies $x\in \Gamma_{d+k-j-i,i} (M(j))$. If the first $n-d+i-(j+1)$ rows are linearly dependent, then $x\notin V(J(M^{c,i}_{d+k-i, j-1}))$; if it were genericity would be violated because 
$x\in \Gamma_{d+k-(j+1)-i,i} (M(j))$. However, $x\in V(J(M_{d+k-i+(j+1},M^{c,i}_{d+k-i, j+1}))$, hence in $\Gamma_{d+k-(j+1)-i,i} (M(j+1))$.

Suppose $j= min( n-d+i-1, k+i)$, say $j=n-d+i-1$; then the matrix of generators for $M_{d+k-i+j}$ consists of a single row, so the first $n-d+i-(j+1)$ rows are linearly independent  so there is only one case.

If $j=k+i\le n-d+i-1$, then the matrix of generators for $M_{d+k-i+j}$ has $n-(k+d)$ columns, while the matrix of generators of $M^{c,i}_{d+k-i, j}$ has  $n-i$ columns, and fewer than $n-i$ rows, hence $J(M^{c,i}_{d+k-i, j})=0$, so again there is one kind of component.\end{proof}

\begin{theorem} We have
$$h_1^{d+k-i}h_2^i\cdot E= \sum_{j=0}^{min(n-d+i-1, i+k)} (-1)^j M_{d+k+j-i}  \cdot   M^{c,i}_{d+k-i, j}.  $$
\end{theorem}
\begin{proof}
If we fix $i$ then by lemma 4.3, we know that each term is the sum of two numbers, one of which is a summand of the previous term except for the first term and last terms. Since this is an alternating sum the common terms cancel. The first term in the sum by lemma 4.2 is the sum of two numbers, one of which is $h_1^{d+k+1-i}h_2^i\cdot E$, and the other is a summand in the second term, so it cancels. Again by lemma 4.3, the last term in the sum consists of a single summand shared with the previous term.
\end {proof}
\begin{cor} Assuming that coordinates are chosen generically,
$$mult_Y(\Gamma_d( N_D(\cX)))$$
$$=\sum_{i=0}^{min(d,n-1)} \sum_{j=0}^{min(n-d+i-1, i+k-1)} (-1)^j { {d+k}\choose i} M_{d+k+j-i}  \cdot   M^{c,i}_{d+k-i, j}. $$
\end{cor}
\begin{proof} This follows from the previous theorem and Cor. 3.7.
\end{proof}

 \begin{cor} Let $s$ denote the minimum of $ \{n-d+i-1, i+k\}$. Assuming that coordinates are chosen generically,
$$mult_Y(\Gamma_d( JM_z(\cX)))=$$
$$e(JM(X),N_D(X))+\sum_{i=0}^{min(d,n-1)} \sum_{j=0}^s (-1)^j { {d+k}\choose i} M_{d+k+j-i}  \cdot   M^{c,i}_{d+k-i, j}. $$
\end{cor}
\begin{proof} This follows from the previous corollary and the multiplicity polar theorem.
\end{proof}
 \begin{cor} Let $s$ denote the minimum of $ \{n-d+i-1, i+k\}$. Assuming that coordinates are chosen generically, and $H$ is not a limiting tangent hyperplane to $X$ at the origin,
$$(-1)^d\chi(X_s)+(-1)^{d-1} \chi((X\cap H)_s)=  $$
$$e(JM(X),N_D(X))+\sum_{i=0}^{min(d,n-1)} \sum_{j=0}^s (-1)^j { {d+k}\choose i} M_{d+k+j-i}  \cdot   M^{c,i}_{d+k-i, j}. $$
\end{cor}
\begin{proof} This follows from the previous corollary and the fact that $$mult_Y(\Gamma_d( JM_z(\cX)))=(-1)^d\chi(X_s)+(-1)^{d-1} \chi((X\cap H)_s).$$ (Cf. \cite{Gaff1} p 130, \cite {BOT}.) \end{proof}

From this formula, in the usual way we can get a formula for the Euler characteristic of a smoothing of a maximal rank determinantal singularity, for the given presentation matrix. If we take hyperplane slices of $X$ all of the above modules specialize. Let $J_q$ denote the Jacobian module of $X\cap l_q$ where $l_q$ is a generic plane of codimension $q$, let $N_q$ denote $N_D(X\cap l_q)$, $M(q)_i$ and $M(q)^{c,i}_{s,t}$ the modules corresponding to the $M_i $ and $M^{c,i}_{s,t}$. Then we get:

\begin{cor} Let $t$ the minimum of $\{d-q, n-1\}$, and  $s$ denote the minimum of $ \{n-d+i-1, i+k\}$. Assuming that coordinates and hyperplanes are chosen generically,
$$(-1)^d\chi(X_s)= \sum_{q=0}^d (-1)^q e(J_q,N_q)+$$
$$\sum_{q=0}^d\sum_{i=0}^t \sum_{j=0}^s (-1)^{q+j}  { {d-q+k}\choose i} M(q)_{d-q+k+j-i}  \cdot   M(q)^{c,i}_{d-q+k-i, j}. $$
\end{cor}
\begin{proof} This follows from the previous corollary by adding the terms from the corresponding formulas, thereby creating a telescoping sum.
\end{proof}

Now we would like to give some examples. 

In the first example, we calculate the multiplicity over $Y$ of the relative  polar curve of a smoothing of the family of space curves $X_l$. This will entail calculating the multiplicity of the pair $(JM(X_l), N_D(X_l)$. This calculation will be facilitated by a theorem of Steven Kleiman (\cite{K2}). We will also need the multiplicity over $Y$ of the polar curve of $N_D(\cX_l)$.

The singularities $X_l$ are defined by the minors of 
$$M_{X_l}=\begin{bmatrix} z&x\\
                                         y&z\\
                                          x^l&y
\end{bmatrix}$$

We assume $l-1$ is not divisible by $3$.

For the formula for the multiplicity over $Y$ of the polar curve of $N_D(\cX_l)$, we have $n=2$, $k=1$, $d=1$, so Corollary 3.7 becomes
$$\sum_{i=0}^{2}  { {2}\choose i} h_1^ih_2^{2-i}=h_1^2+2h_1h_2+h_2^2. $$

By Theorem 4.4 we have:
$$h_1^2= M_{d+k}  \cdot   M^{c,0}_{d+k, 0} =M_{2}\cdot M^{c,0}_{2,0}$$
$$h_1h_2= M_{d+k-1}  \cdot   M^{c,1}_{d+k-1, 0} - M_{d+k}  \cdot   M^{c,1}_{d+k-1, 1}=M_1\cdot  M^{c,1}_{1, 0}-M_{2}  \cdot   M^{c,1}_{1, 1}$$ 

In this example our matrix is generic enough; no operations on rows and columns are needed.

We have  $M_{2}\cdot M^{c,0}_{2,0}$ is the colength of the elements of the first row and the determinant of the last two rows which is the colength of $(x,z,y^2-xz)$ which is 2.

Denoting the $(i,j)$ entry of $M$ by $a_{i,j}$, we have $M_1\cdot  M^{c,1}_{1, 0}$ is the colength of the ideal generated by $(a_{1,1}a_{2,2}-a_{1,2}a_{2,1},  a_{1,2}, a_{3,2})=(a_{1,1}a_{2,2}, a_{1,2}, a_{3,2})=(z^2,x,y)$ which is $2$. 
We also have  $M_{2}  \cdot   M^{c,1}_{1, 1}$ is the colength of the ideal generated by $(a_{1,1}, a_{1,2}, a_{3,2})=(z,x,y)$, which is $1$. 
this means that for all $l$ the multiplicity over $Y$ of the polar curve of $N_D(\cX_l)=2\cdot (2-1)+2=4$, which is $m+1$ since $m$ for these curves is $3$.

Notice that in general we see that the ideal $(a_{1,1}a_{2,2}, a_{1,2}, a_{3,2})$ has two groups of components; those of $ (a_{1,1}, a_{1,2}, a_{3,2})$ and of $(a_{2,2}, a_{1,2}, a_{3,2})$. It is those of the second type that make up the $\Gamma_{1,1}$. Those of the first type are subtracted in the second term.

We must now calculate $e(JM(X_l),N_D(X_l))$. We will show it is $2l-2$.

Tracing through the connection between the deformations of the presentation matrix and the elements of $N_D$, we get that a matrix of generators $M[N_D]$, of $N_D$ is:

$$M[N_D]=\begin {bmatrix} z&-x&0&-y&z&0\\
0&y&-z&0&-x^l&y\\
y&0&-x&-x^l&0&z\end{bmatrix}.$$

Because $X$ is weighted homogeneous with weights $(3,2l+1,l+2)$, $X$ has a parameterization by
$r(t)=(t^3,t^{2l+1},t^{l+2})$. For our computations it will be convenient to work over $\cO_1$ by pulling back our modules by $r^*$.

Pulling back the matrix of generators $M[N_D]$, we obtain:

$$r^*(M[N_D])=\begin {bmatrix} t^{l+2}&-t^3&0&-t^{2l+1}&t^{l+2}&0\\
0&t^{2l+1}&-t^{l+2}&0&-t^{3l}&t^{2l+1}\\
t^{2l+1}&0&-t^3&-t^{3l}&0&t^{l+2}\end{bmatrix}.$$

As this matrix has generic rank $2$, $r^*(N_D)$ can be generated freely by $2$ generators, so a matrix of generators $R_N$ of $r^*(N_D)$ with a minimal number of columns is

$$R_N=\begin {bmatrix} -t^3&0\\
t^{2l+1}&-t^{l+2}\\
0&-t^3\end{bmatrix}.$$

A calculation shows that $r^*(JM(X))$ is generated by the columns of:

$$R_{JM}=\begin {bmatrix} -t^3&2t^{l+2}\\
2t^{2l+1}&-t^{3l}\\
t^{l+2}&t^{2l+1}\end{bmatrix}.$$

Note that $$R_{JM}=R_N \begin {bmatrix}1&-2t^{l-1}\\
-t^{l-1}&-t^{2l-2}
\end{bmatrix}.$$

Since $r^*(N_D)$ is freely generated, $e(r^*(JM(X)),r^*(N_D))$ is the colength of the submodule of $\cO^2_1$ generated by the columns of the $2\times 2$ matrix in the last equation. This is the colength of the determinant of the matrix of generators which is $2l-2$ as claimed. By the result of Kleiman mentioned above, $e(r^*(JM(X)),r^*(N_D))=e(JM(X),N_D)$.

 By results of Watanabe et al \cite{NNW},  $\delta(X_l)=l$. Hence the Milnor number is $2\delta=2l$, so $e(JM(X_l),N_D(X_l))=\mu-2$, and 
so  $e(JM(X_l),N_D(X_l))$ plus the multiplicity over $Y$ of the polar curve of $N_D(\cX_l)$ is $\mu-2+m+1=\mu+m-1$, which is  the multiplicity over $Y$ of the relative  polar curve of a smoothing of the family of space curves $X_l$ as predicted.

\section{Consequences for equisingularity--the Whitney conditions, $W_f$, $A_l$ and $A_f$ conditions.}

In this section we will apply the results from previous sections to provide criteria for various equisingularity conditions. We first review the conditions we will study, and recall some basic properties.

The conditions we will study are concerned with the relation between tangent planes 
plane to a stratum and limiting
tangent hyperplanes from higher-dimensional strata, so we need a way of
measuring distance between linear spaces.
  Suppose $A$, $B$ are linear subspaces at the origin  in ${\mathbb C}^{N} $,
then define the distance from
$A$ to $B$ as:
                         $${ {\rm{dist }}}(A,B) =
\mathop {\sup}\limits_{\begin{matrix} {u\in {B^{\perp}}-\{0\}}\\{v\in
{A-\{0\}}}\end{matrix}} {{|(u,v)|} \over {\left\| u
\right\|\left\| v \right\|}}.$$

        In the applications $B$ is the ``big'' space and $A$ the ``small''
space.
        (Note that ${\rm{dist }}(A,B)$ is not in general the same as
${\rm{dist }}(B,A)$.)

We recall Verdier's condition $W$. Suppose
$Y\subset
\Bar X$, where $X,Y$ are strata in a stratification of an analytic space,
and $
{\rm{dist }}(TY_0, TX_x) \le C
{\rm{dist }}(x,Y)$ for all $x$ close to $Y$.   Then the pair $(X,Y)$
satisfies {\it Verdier's condition} $W$ at $0 \in Y$.

The Whitney conditions are also used to describe the incidence relation of two strata; however Teissier proved  condition $W$ is  equivalent to these two Whitney
conditions in the complex analytic
case \cite{T-2},V.1.2,
 so we will use the two terms interchangeably.

The following result connects theory of integral closure and the condition $W$.

\begin{Prop}  ({\cite{G-2}, Theorem 2.5}) Suppose $(X^{d+k},0) \subset
({\mathbb C}^{n+k},0)$ is an equidimensional complex
analytic set, $X = F^{-1}(0)$, $Y$ a smooth subset of $X$,
$F: {\mathbb C}^{n+k}\to {\mathbb C}^p$, coordinates  chosen so
that ${{\mathbb C}}^k \times {0} = Y$, $m_Y=(z_1, ...,z_n)$ denoting the ideal
defining $Y$,
$F$ defines $X$ with reduced structure.   Then, ${{\partial F} \over
{\partial s}}\in \overline
{m_{n}JM(F)}$ for all tangent vectors ${\partial } \over {\partial s}$ to
$Y$ if and only if W holds for
$(X_0,Y)$.\end{Prop}
\begin{proof} See  ({\cite{G-2}, Theorem 2.5})\end{proof}

If $f$ is the germ of a complex analytic mapping defined on the closure  of
a stratum $X$, then it is useful
to have a notion of condition $W$ relative to $f$; this is obtained from
the above definition of  $W$ by  replacing the tangent
space of $X$  by the tangent space of the fiber of $f$ at points where $f$
is a submersion onto its image.
The {\it $W_f$ condition}  holds  for the pair $(X,Y)$  when the new
condition holds;  when
it holds with some unspecified exponent on the term ${\rm{dist }}(x,Y)$, we say that {\it Thom's condition
$A_f$} holds. (It can be shown using the theory of integral closure, that if every limiting tangent plane to the fibers of $f$ contains $TY_y$, along $Y$, then the distance condition holds with some exponent.)

If we have a maximal rank isolated determinantal singularity $X$, then the invariants of this section take the form 
$e(M,N_D)+ mult _Y\Gamma (N_D)$, where $Y$ is the base of the smoothing defined by the matrix defining our singularity. In this first part, $M$ will be the Jacobian module of $X$ or $mJM(F)$. We denote the sum $e(M,N_D)+ mult _Y\Gamma (N_D)$ by $e_{\Gamma}(M,N_D)$; by the results of the previous section, this is depends only on the presentation matrix of $X$, hence is independent of any family deforming $X$ induced from a deformation of the presentation matrix. 

Here is our result for condition $W$. The result is a modification of theorem 5.7 of \cite{Gaff1} to fit the context of this paper. 
Since in our situation we assume that $Y$ is in the cosupport of $JM(F)$,
and  $m_YJM(F)=JM(F)$ off $Y$, it follows that
$e(m_y JM(F_y),N_X,(z,y))=e(JM(F_y, N_X,(z,y)))$, $z\ne0$.

\begin{theorem}   Suppose $( X^{d+k},0) \subset  (\mathbb C^{n+k},0)$, 

$X = F^{-1}(0)$,  $F:{\mathbb C}^{n+k}\to{\mathbb C}^p$, $Y$ a smooth subset of $X$,
coordinates  chosen so
that ${\mathbb C}^k \times {0} = Y$, $F$ induced from a deformation of the presentation matrix of  $X_0$, $X$ equidimensional with equidimensional fibers, of expected dimension.

A) Suppose $X_y$ are isolated, maximal rank determinantal 
singularities, suppose the singular set of $X$ is $Y$.
Suppose $e_{\Gamma}(m_yJM(F_y),N_D)(y))$ is independent of $y$. 
 Then the union of the singular points
of $X_y$
 is $Y$, and the pair of strata
$(X-Y,Y)$ satisfies condition $W$.

B) Suppose $\Sigma(X)$
 is $Y$  and  the pair $(X-Y,Y)$ satisfies condition $W$.  Then  
 $e_{\Gamma}(m_yJM(F_y),N_D(y))$ is independent of $y$.\end{theorem}

\begin{proof} Again note that $JM(F)\in H_0(JM_z(F))$. 

Now we prove A). We can embed the family in a restricted versal unfolding with smooth base $\tilde Y^l$. Consider the polar variety of $JM_z(F)$ of dimension $l$, and the degree of its projection to $\tilde Y^l$ along points of $Y$. The hypothesis on $e_{\Gamma}$ implies by the multiplicity polar theorem that this degree is constant over $Y$. In turn this implies that the polar variety over $Y$ does not split, hence the polar of the original deformation is empty. This then implies that the PSID holds.

Since $W$ holds generically, by the PSID, it follows that 
$${JM_Y(F)}\subset\?{m_YJM_z(F)}.$$ This implies that $\?{JM(F)}\subset\?{JM_z(F)}$.
 Hence the union of the singular points
of $F_y$ which is the cosupport of $\?{JM_z(F)}$ is equal to the cosupport of $\?{JM(F)}$
which is $Y$.  Then the inclusion ${JM_Y(F)}\subset\?{m_YJM_z(F)}$ implies $W$ for $(X-Y,Y)$. (Cf \cite{G-2})

Now we prove B).  $W$ implies ${JM_Y(F)}\subset\?{m_YJM_z(F)}$ which implies that $\?{m_YJM(F)}=\?{m_YJM_z(F)}$.
We know by \cite{T-2} that condition $W$ implies that the fiber dimension of the exceptional divisor of
$B_{m_Y}(C(X))$ over each point of $Y$ is as small as possible.
The integral closure condition $\?{m_YJM(F)}=\?{m_YJM_z(F)}$ implies that the same is true for
$B_{m_Y}(\Projan{{{\mathcal R}}} (JM_z(F)))$. This implies that the polar of $m_YJM_z(F)$ is empty, hence by the multiplicity polar formula
the invariant $e_{\Gamma}(m_yJM(F_y),N_D)(y))$ is independent of $y$.\end{proof}

There is a nice geometric interpretation of the number $e(mJM(F_y))$ 
which we now describe. We denote the multiplicity of the polar 
variety of 
$X_y,0$ of dimension $i$ by $m^i(X_y)$.
\begin{theorem} Suppose  $N$ any submodule  in $\cO^p_{X_y,0}$ such 
that $\dim_{\mathbb C}N/JM(F_y)<\infty$, then
$$e(mJM(F_y),N)=e(JM(F_y),N)+\sum\limits_{i=1}^{d}{{n-1}\choose i } 
m^i(X_y).$$
\end{theorem}
\begin{proof} This is exactly the content of the formula in Theorem 9.8 (i) p.221 
\cite{KT}.
\end{proof}
\begin{cor} We have:
$$e_{\Gamma}(mJM(F_y),N_D(y))=e_{\Gamma}(JM(F_y),N_D(y))+\sum\limits_{i=1}^{d}{{n-1}\choose i } 
m^i(X_y).$$\end{cor}

\begin{proof} Simply add $mult _Y\Gamma_d (N_D)$ to both sides of the formula of the last theorem.\end{proof}
 
On earlier work on ICIS it was possible to drop the hypothesis $S(X)=Y$ in part A) of theorem 5.2. There since $m_d(X^d(y))$, was non-zero at any singular point, the independence of $m_d(X^d(y))$ from $y$ prevented the singular locus from splitting. In order to strengthen 5.2A, we would like to know when $e_{\Gamma}(JM(F_y),N_D)$ is  non-zero. As a first step we prove a result about when $mult _Y\Gamma_d (N_D)$ is  non-zero.

\begin{Prop} Assume the presentation matrix of $X^d\subset {\mathbb C}^q$ has size $(n+k,n)$, all of whose entries are zero at zero. Then the polar curve of $N_D$ is  non-empty if and only if 
$d+k+1=q\le 2n+k-1$.
\end{Prop}
\begin{proof} We assume $\mathcal X$ is a smoothing of $X$. Then $\mathcal X$ is itself an isolated determinantal singularity. The non-emptiness of the polar curve is equivalent to the dimension of the fiber of ${\mathcal X}_M$ over the origin having maximal dimension $q-1$. This is maximal because the dimension of ${\mathcal X}_M$ is $q$. The fiber of ${\mathcal X}_M$ over the origin is inside ${\mathbb P}^{n+k-1}\times {\mathbb P}^{n-1}$. Thus a necessary condition for the fiber over the origin to be of dimension $q-1$ is 
for $q-1\le  (n+k-1)+(n-1)= 2n+k-2$ or $q\le  2n+k-1$.

Assume $q\le  2n+k-1$. We must show the fiber over the origin has dimension $q-1$. 

Consider the set of equations of form $(T_1,\dots,T_n)[M]^t=0$ and of form $(S_1,\dots,S_{n+k})[M]=0$, where $(T_1,\dots,T_n)$ (resp.  $(S_1,\dots,S_{n+k})$) are coordinates on ${\mathbb P}^{n-1}$ (resp. ${\mathbb P}^{n+k-1}$). Off $0\times {\mathbb P}^{n+k-1}\times {\mathbb P}^{n-1}$, these are a set of defining equations for ${\mathcal X}_M$. Denote by $V$ the zero set of these equations. Locally at points of $0\times {\mathbb P}^{n+k-1}\times {\mathbb P}^{n-1}$, we can cut down on the number of equations needed. To illustrate the idea, suppose we are working at $(1,0,\dots,0)\times (1,0,\dots,0)$. The equations of the first type become $(1,\dots,T_n)[M]^t=0$. If the equations hold at a point then the first column of $[M]$ must be in the span of the other columns. So it suffices to use only $n-1$ equations of the form  $(S_1,\dots, S_{n+k})[M]=0$ to cut out the vanishing of the original set of equations. Locally then we use $(n+k)+(n-1)$ equations to cut out $V$. This implies that each component of $V$ must have codimension at most $2n+k-1$ or dimension at least $[(q+1)+(n-1)+(n+k-1)]-(2n+k-1)=q$. If  $q-1=  2n+k-2$, then although $0\times {\mathbb P}^{n+k-1}\times {\mathbb P}^{n-1}$ lies in $V$, it cannot be a component of $V$, hence is in  ${\mathcal X}_M$, and the dimension of the fiber over $0$ of  ${\mathcal X}_M$ is maximal.

So we may suppose $q-1<  2n+k-2$. In this case, since ${\mathcal X}$ has an isolated singularity, $V$ has two components  $0\times {\mathbb P}^{n+k-1}\times {\mathbb P}^{n-1}$ and ${\mathcal X}_M$. By Grothendieck's connectedness theorem (\cite {GR}, XIII2.1) it follows that the two components intersect in a set of dimension $q-1$, which finishes the proof.

\end{proof} 

\begin{cor} Assume the presentation matrix of $X^d\subset {\mathbb C}^q$ has size $(n+k, n)$, all of whose entries are zero at zero. Then $m_d(X^d)$  is  non-zero if 
$q\le 2n+k-1$.
\end{cor}
\begin{proof}
If the polar of $N_D$ is non-empty, since $e(JM(X),N_D(0))$ is non negative, and $m_d(X^d)$ is the sum of these two terms, then  $m_d(X^d)$ is non-zero.
\end{proof}

The next example shows the necessity for the hypothesis about the singular set of our family. 

Damon-Pike \cite{D-P} looked at invariants connected with isolated singularities in the case of $2\times 3$ matrices. They viewed the matrices as maps $F$ from ${\mathbb C}^p\to Hom({\mathbb C}^3,{\mathbb C}^2)$. Their techniques allowed them to calculate the reduced Euler characteristic $b_3-b_2$ of the smoothings for many determinantal singularities of this type with $p=5$. In particular, they showed that the reduced Euler characteristic was $-1$ for  the smoothings of the singularities associated with the following maps:

$$F_k=\begin {bmatrix} w&y&x\\
z&w&y+v^k
\end{bmatrix} \hskip 2em 1\le k\le 5.$$

Denote the associated singularity by $X_k$. We  take a generic linear section of $X_k$ by restricting $F_k$ to a generic hyperplane and considering the associated singularities of the restricted map. If we do this by taking $v=0$ for example, we get a surface singularity of type associated with the map:

$$G=\begin {bmatrix} w&y&x\\
z&w&y
\end{bmatrix} \hskip 2em 1\le k\le 5.$$

The Tjurina number, $\tau$, of this is $2$ (\cite {FN}); Pereira and Ruas (\cite {P-R}) showed that for this type of singularity, the Milnor number of the smoothing, and the only non-zero Betti number appearing in the reduced Euler characteristic, satisfies $\tau=\mu+1$. So here $\mu=1$. 

Since $m_d(X^d)=(-1)^d\chi(H_*(X_s))+(-1)^{d-1}\chi(H_*((X\cap H)_s)$,
we get $m_3(X_k)=b_3-b_2+\mu=0$.

So for these singularities the relative Jacobian module of the smoothing has no polar curve.

Now consider the family of sets $X_t$ defined by the family of maps:

$F_{(2,3),t}=\begin {bmatrix} w&y&x\\
z&w&y+tv^2+v^3
\end{bmatrix}.$

For each value of $t$, $m_3(X_t)=0$. Further, all of these singularities have equivalent hyperplane slices; this immediately means that the multiplicity and the multiplicity of the codimension 1 polars of the $X_t$ are the same, as these multiplicities are invariant under slicing by generic hyperplanes. It also implies that the multiplicity of the family of polar curves of the $X_t$ is constant as well. We will recall this.  Consider the union of the polar curves; when we intersect this set with a generic hyperplane which contains the parameter axis, we get the parameter axis and a curve. This curve is the relative polar curve of the total space of $X\cap H$. It is non-empty if and only the multiplicity of the union of curves changes with the parameter. In our situation, the relative polar curve of $X\cap H$ is empty because $X\cap H$ is a family of analytically equivalent surfaces. (More than we need.) Hence the multiplicity of the polar curves of $X_t$ are constant.

However, the singular locus of the total space is given by the zeros of the entries of $F_{(2,3),t}$. These are
$x=y=w=z=0$, $v^2(t+v)=0$, so the locus splits at $t=0$. Thus $X-Y$ is not a stratum in this case, although $(X-S(X),Y)$ is a pair of strata which satisfy the Whitney conditions at the origin, as there is no polar curve, and the other polar multiplicities are constant. For this type of example, a necessary and sufficient condition to prevent the splitting is to ask that the multiplicity of the ideal of entries of the $F_t$ be independent of $t$.

In the examples, on the list of Damon and Pike, the ideal of entries is generated by 5 elements; hence the above independence of parameter is necessary and sufficient for no splitting of the singular set of a family.

We can use Corollary 5.6 to strengthen 5.2A when the hypotheses of 5.6 apply.

\begin{Prop} Suppose in Theorem 5.2A, $d=1$ or $d=2$. If the invariant $e_{\Gamma}(m_yJM(F_y),N_D)(y)$ is independent of $y$, then
$X-Y$ is smooth, and the pair of strata
$(X-Y,Y)$ satisfies condition $W$.\end{Prop}
\begin{proof} If $d=2$, then for any $n$, a maximal rank determinantal singularity $V$ with presentation matrix of size $(n,n+k)$ has the property that $e_{\Gamma}(JM(V),N_D(V))$ is  positive. Thus the independence of the invariant $e_{\Gamma}(m_yJM(F_y),N_D)(y)$ from parameter implies that splitting cannot occur. \end{proof} 

We now turn to the two relative conditions $\Af$ and $\Wf$. 

The result for $\Af$ is a re-tuning of theorem 5.6 of \cite{Gaff1} for our situation.

\begin{theorem}   Suppose $( X^{d+k},0) \subset  (\mathbb C^{n+k},0)$, is a determinantal singularity with presentation matrix $M$.

$X = G^{-1}(0)$,  $G:{\mathbb C}^{n+k}\to{\mathbb C}^p$, $Y$ a smooth subset of $X$,
coordinates  chosen so
that ${\mathbb C}^k \times {0} = Y$, $X$ equidimensional with equidimensional fibers of the expected dimension, $X$ reduced.

Suppose $F\:(X,Y)\to \mathbb C$,  $F\in m^2_Y$, $Z=F^{-1}(0)$.

A) Suppose $X_y$ and $Z_y$ are isolated
singularities, suppose the singular set of $F$ is $Y$.
Suppose $e_{\Gamma}(JM(G_y;F_y),\cO_{n+k}\oplus N_D(y))$ is independent of $y$. 
 Then the union of the singular points
of $F_y$
 is $Y$, and the pair of strata
$(X-Y,Y)$ satisfies Thom's \AF \hskip 2pt condition.

B) Suppose $\Sigma(F)$
 is $Y$ or is empty, and  the pair $(X-Y,Y)$ satisfies Thom's \AF\hskip 2pt
condition.  Then   $e_{\Gamma}(JM(G_y;F_y),\cO_{n+k}\oplus N_D(y))$ is independent of $y$.\end{theorem}

\begin{proof} The condition that  $X_y$ and $Z_y$ are isolated singularities implies that the integral closure of $JM(G_y;F_y)$ contains $\cO_{n+k}\oplus N_D(y)$ for all $y$, so the multiplicity of the pair $(JM(G_y;F_y),\cO_{n+k}\oplus N_D(y))$
 is well defined.  For the module  $\cO_{n+k}\oplus N_D$,  $\Projan\cR(\cO_{n+k}\oplus N_D)$ is the join of a point with $\Projan \cR(N_D)$ in $\mathbb P^n$. This implies that the polar variety of $\cO_{n+k}\oplus N_D$ of dimension $K$ is the same as the polar variety of dimension $k$ of $N_D$. With this observation the proof of Theorem 5.6 in \cite {Gaff1} goes through.\end{proof}

The hypothesis $F\in m^2_Y$ is technical, and is used to ensure the components of the relative conormal over $Y$ have maximal dimension. Without this hypothesis it seems necessary to assume that $W_A$ holds for the pair $(X-Y,Y)$.

The result  and proof for $\WF$ is similar.

\begin{theorem}   Suppose $( X^{d+k},0) \subset  (\mathbb C^{n+k},0)$, is a determinantal singularity with presentation matrix $M$.

$X = G^{-1}(0)$,  $G:{\mathbb C}^{n+k}\to{\mathbb C}^p$, $Y$ a smooth subset of $X$,
coordinates  chosen so
that ${\mathbb C}^k \times {0} = Y$, $X$ equidimensional with equidimensional fibers of the expected dimension, $X$ reduced.

Suppose $F\:(X,Y)\to (\mathbb C,0)$, $Z=F^{-1}(0)$.

A) Suppose $X_y$ and $Z_y$ are isolated
singularities, suppose the singular set of $F$ is $Y$.
Suppose $e_{\Gamma}(m_YJM(G_y;F_y),\cO_{n+k}\oplus N_D(y))$ is independent of $y$. 
 Then the union of the singular points
of $F_y$
 is $Y$, and the pair of strata
$(X-Y,Y)$ satisfies the \WF \hskip 2pt condition.

B) Suppose $\Sigma(F)$
 is $Y$ or is empty, and  the pair $(X-Y,Y)$ satisfies the \WF\hskip 2pt
condition.  Then   $e_{\Gamma}(m_YJM(G_y;F_y),\cO_{n+k}\oplus N_D(y))$ is independent of $y$.\end{theorem}

\begin{proof} Note that we can drop the hypothesis $F\in m^2_Y$. This is because the construction of $ \Projan \cR (m_YJM_z(G;F))$ involves blowing up by the pullback of $m_Y$ to $\Projan\cR(JM_z(G;F))$; this has the effect of ensuring that all components of the inverse image of $Y$ have maximal dimension. Then the proof is similar to that of Corollary 4.14  of \cite{G-P-H}, keeping in mind the remarks made in the proof of the last theorem. \end{proof}

\section{Two Whitney equisingular families with generic members with different smoothings}

 In this section we present an example of a singularity which is a member of two Whitney equisingular families, whose generic elements have topologically distinct smoothings. The invariant of the last section $e_{\Gamma}(JM(X),N_D)$ takes on different values for each family, though independent of parameter within a family. This example shows that it is impossible to find an invariant which depends only on an analytic space $X$ with an isolated singularity, whose value is independent of parameter for all Whitney equisingular deformations of $X$, and which is determined by the geometry of a smoothing of $X$.

 The two families lie in different landscapes of $X$. We will see that our invariant $e_{\Gamma}(JM(X),N_D)$ gives different values for each family because the generic element in each landscape has different topology. 
 
Our example is constructed from 
the following example, taken from Wahl \cite[pg. 52]{Wahl}. Consider the rational surface singularity $(X,0) \subset (\mathbb{C}^{5},0)$ given by the vanishing of the $2$ by $2$ minors of
\begin{equation}\label{central fiber}
\begin{pmatrix}
x_1 & x_2 & x_3 &  x_4 \\
x_2 & x_3 & x_4 & x_1+x_5^{2}
\end{pmatrix}
\end{equation}
The miniversal base space has components of dimensions $4,6$ and $8$ and $\dim T_{X}^{1}=10$. The versal determinantal deformation of dimension $8$ is given by
\begin{equation}
\begin{pmatrix}\label{8-dim}
x_1 & x_2 & x_3 &  x_4 \\
x_2+t_1+t_6x_5 & x_3+t_2+t_7x_5 & x_4+t_5+t_8x_5 & x_1+x_5^{2}+t_4+t_{10}x_3
\end{pmatrix}.
\end{equation}
We denote the eight dimensional base of this deformation by $V_1$.
The versal family over a component of dimension 4 is obtained from the $2$ by $2$ minors of the following matrix
\begin{equation}
\begin{pmatrix}\label{4-dim}
x_1 & x_2 & x_3 \\
x_2 & x_3+t_5+t_9x_5 & x_4\\
x_3 & x_4 & x_1+x_5^{2}+t_4+t_{10}x_3
\end{pmatrix}.
\end{equation}
We denote the four dimensional base of this deformation by $V_2$.
Our goal is to deform $(X,0)$ over $V_1$ and $V_2$ so that
the resulting one-parameter families are Whitney equisingular. Moreover, we want to show that
the two landscapes of $(X,0)$ corresponding to $V_1$ and $V_2$ are distinct. To do this it is enough to prove that the Euler characteristics of the smooth generic fibers of the one-parameter deformations of $(X,0)$ are different. We will make use the following formula (see \cite{Gaff1}) specialized to our setting
\begin{equation}\label{polar-Euler}
\mathrm{mult}_{S_i}(\Gamma_2(\mathrm{Jac}_z(\mathcal{X}_{S_i}))=\chi(X_{s_i})-\chi((X \cap H)_{s_i})
\end{equation}
where $S_i$ for $i=1,2$ are the parameter spaces of smoothings that lie in a $V_i$ respectively, $\mathcal{X}_{S_i}$ is the total space of the family over $Y$, $\mathrm{Jac}_z(\mathcal{X}_{S_i})$ is the relative Jacobian module, i.e. the module obtained by considering partials with the fiber coordinates only, $\chi(X_{s_i})$ is the Euler characteristic of the generic fiber $X_{s_i}$ and $\chi((X \cap H)_{s})$ is the Euler characteristic  of a smoothing of a hyperplane slice of $X$ over $V_i$, and $H$ is a hyperplane that is not a limiting tangent hyperplane of $(X,0)$ at $0$.

Assume $H$ is chosen so that $X \cap H$ is reduced. Then $\chi((X \cap H)_s)$ is constant regardless of which component of the miniversal base space we use, because $X \cap H$ is a reduced curve (see part (2) of Theorem 4.2.2 in \cite{BG}). So it suffices to show that the polar multiplicity in (\ref{polar-Euler}) is different across the two smoothing components.

The following observation will prove useful on several occasions throughout the remainder of this section.

\begin{Prop} The hyperplane $H_{x_5}$ defined by the vanishing of $x_5 =0$ in $\mathbb{C}^{5}$ is not a limiting tangent hyperplane at $0 \in X$.
\end{Prop}
\begin{proof} Note that by identifying each tangent hyperplane at smooth points of $(X,0)$ with its corresponding conormal vector, it is enough to show that the vector $$e_5 :=[0,0,0,0,1]$$ is not a limit of conormal vectors at smooth points of $(X,0)$. Furthermore, it is enough to check this along curves $\phi:(\mathbb{C},0) \rightarrow (X,0)$. The matrix
\begin{equation}\label{module K}
[K] := \begin{pmatrix}
-x_3 & 2x_2 & -x_1 & 0 & 0 \\
x_4  & -x_3 & -x_2 & x_1 & 0 \\
2x_1+x_5^2 & -x_4 & 0 & -x_2 & 2x_5x_1
\end{pmatrix}
\end{equation}
is obtained from the Jacobian matrix of $(X,0)$ by deleting three of its rows. Away from $\mathbb{V}(x_1x_5)$ the matrix is of maximal possible rank $3$. If $\phi$ lies in $\mathbb{V}(x_1x_5)$, then all entries in the column vector corresponding to the partials with respect to $x_5$ vanish, so all limits of conormal vectors along $\phi$ have a vanishing fifth component. Therefore, $e_5$ is not a limit of conormal vectors at points from the locus $\mathbb{V}(x_1x_5)$. Now suppose that $\phi$ does not lie in $\mathbb{V}(x_1x_5)$. Then we can assume that
$\phi$ meets $\mathbb{V}(x_1x_5)$ only at $0$
after possibly replacing $(X,0)$ with a smaller neighborhood.

Let $t$ be a generator for the maximal ideal of $\mathcal{O}_{\mathbb{C}}$. Any conormal vector along $\phi$  can be represented as a linear combination of the rows of $[K]$ with coefficients from $\mathcal{O}_{\mathbb{C}}$, i.e. we can write
\begin{equation}
v(t) = a_1(t)v_1(t)+a_2(t)v_2(t)+a_3(t)v_3(t),
\end{equation}
where $a_i(t) \in \mathcal{O}_{\mathbb{C}}$, $v_i(t) := v_i \circ \phi$ where $v_i$ are the rows of $[K]$. To prove that
$$\lim_{t \rightarrow 0}\frac{v(t)}{t^{k}} \neq e_5,$$
where $k$ is the minimum of the orders of $t$ in the components of $v(t)$, it suffices to show the order of $t$ in the fifth  component of $v(t)$ is strictly larger than one of the remaining components of $v(t)$. Suppose the contrary. Set $x_i(t)=t^{\alpha_i}\hat{x_i}$ and $a_i(t) = t^{\lambda_i}\hat{a_i}$ where $\hat{x_i}$ and $\hat{a_i}$ are units in $\mathcal{O}_{\mathbb{C}}$. The vanishing of either $x_i$ or $a_i$ for some $i$ simplifies the problem considerably, so we can assume that each $\alpha_i$ and $\lambda_i$ is a finite nonnegative integer. The  defining equations for $(X,0)$ give us
\begin{align}
\alpha_1+\alpha_3 & = 2\alpha_2 \label{ord.first} \\
\alpha_2+\alpha_4 & =2\alpha_3 \label{ord.second} \\
\alpha_1+\alpha_4 & =\alpha_2+\alpha_3 \label{ord.third}
\end{align}
and the following three possibilities for $\alpha_i$
\begin{displaymath}
\alpha_1>\alpha_2>\alpha_3>\alpha_4, \ \ \text{or} \ \ \alpha_4>\alpha_3>\alpha_2>\alpha_1  \ \ \text{and} \ \ \mathrm{ord}_{t}(2x_1+x_5^{2})=\alpha_1,
\end{displaymath}
or
$$ \alpha_1=\alpha_2=\alpha_3=\alpha_4.$$
Suppose $\alpha_1>\alpha_2>\alpha_3>\alpha_4$. The last two columns of $[K]$ give us
\begin{displaymath}
\mathrm{ord}_{t}(A_3(t)x_5(t)x_1(t))<\mathrm{ord}_{t}(x_1(t)A_2(t)-x_2(t)A_3(t)).
\end{displaymath}
 However, $\mathrm{ord}_{t}(A_3(t)x_5(t)x_1(t))>\mathrm{ord}_{t}(x_2(t)A_3(t))$ because $\alpha_1>\alpha_2$.
Therefore, $\mathrm{ord}_{t}(A_2(t)x_1(t))=\mathrm{ord}_{t}(x_2(t)A_3(t))$, i.e.
\begin{equation}\label{eq.l1}
\lambda_2+\alpha_1=\lambda_3+\alpha_2.
\end{equation}
and hence $\lambda_3>\lambda_2$. By the same token comparing the third and the fifth components of $v(t)$ we get
\begin{equation}\label{eq.l2}
\lambda_2+\alpha_2=\lambda_1+\alpha_1
\end{equation}
We claim that (\ref{eq.l1}) and (\ref{eq.l2}) imply that the order of $t$ in the summands in the second component of $v(t)$ are equal. To see this, it is enough to check that
$$\lambda_3+\alpha_4=\lambda_2+\alpha_3=\lambda_1+\alpha_2.$$
Indeed, $\lambda_3+\alpha_4=\lambda_2+\alpha_3$ because by (\ref{eq.l1}) $\lambda_3+\alpha_2=\lambda_2+\alpha_1$
and by (\ref{ord.third}) we obtain $\alpha_1-\alpha_2=\alpha_3-\alpha_4$. Also, $\lambda_2+\alpha_3=\lambda_1+\alpha_2$
because by (\ref{eq.l2}) $\lambda_2+\alpha_2=\lambda_1+\alpha_1$ and by (\ref{ord.second}) we get $2\alpha_2=\alpha_1+\alpha_3$. Thus, the orders of $t$ in each summand of the second, third and fourth components of $v(t)$ are the same. Therefore, by comparing the orders of each of these components with fifth component of $v(t)$, we get that each of the expressions
$2\hat{a_1}\hat{x_2}-\hat{a_2}\hat{x_3}-\hat{a_3}\hat{x_4},$ $\hat{x_1}\hat{a_1}+\hat{x_2}\hat{a_2},$ and $\hat{x_1}\hat{a_2}-\hat{x_2}\hat{a_3}$ is divisible by $t$. Reduce modulo $t$ and identify $\hat{x_i}$ and $\hat{a_i}$ with their images in $\mathbb{C}$. We get
\begin{gather}
\frac{\hat{x_2}}{\hat{x_1}} = -\frac{\hat{a_1}}{\hat{a_2}}=\frac{\hat{a_2}}{\hat{a_3}} \label{frac}, \\
2\hat{a_1}\hat{x_2}-\hat{a_2}\hat{x_3}-\hat{a_3}\hat{x_4}=0 \label{quadratic}.
\end{gather}
The original set of equations for $(X,0)$ gives
$$\hat{x_3}=\frac{\hat{x_2}^{2}}{\hat{x_1}} \ \ \text{and} \ \ \hat{x_4}=\frac{\hat{x_2}\hat{x_3}}{\hat{x_1}} =\frac{\hat{x_2}^{3}}{\hat{x_1}^{2}}.$$
Substituting $\hat{x_3}$and $\hat{x_4}$ with the corresponding expressions in (\ref{quadratic}), and factoring out $\hat{x_2}$ we get
$$\hat{a_3}\left(\frac{\hat{x_2}}{\hat{x_1}}\right)^{2}+\hat{a_2}\left(\frac{\hat{x_2}}{\hat{x_1}}\right)-2\hat{a_1}=0.$$
Since $\hat{a_1}\hat{a_3}=-\hat{a_2}^{2}$ we obtain
$$ \frac{\hat{x_2}}{\hat{x_1}}= \frac{\hat{a_2}}{\hat{a_3}} \left(\frac{-1 \pm \sqrt{7}i}{2}\right)$$
which contradicts (\ref{frac}).

Next, suppose that $\alpha_4>\alpha_3>\alpha_2>\alpha_1$. By comparing the first and fifth components of $v(t)$ we get $$\mathrm{ord}_{t}(a_3(t)x_5(t)x_1(t)) > \mathrm{ord}_t(a_3(t)(2x_1(t)+x_5(t)^{2})$$ because $\mathrm{ord}_t(2x_1+x_5^2)=\alpha_1$ and $\alpha_5>0$. This forces $$\lambda_3+\alpha_1=\min\{\lambda_2+\alpha_4,\lambda_1+\alpha_3\}.$$ Assume $\lambda_3+\alpha_1=\lambda_2+\alpha_4$. Then $\lambda_3>\lambda_2$ because $\alpha_4>\alpha_1$. From the fourth and fifth components of $v(t)$ we get
$\lambda_3+\alpha_1>\lambda_2+\alpha_1$. This forces the orders of each of the summands in the fourth component of $v(t)$ to be the same, i.e. $\alpha_1+\lambda_2=\lambda_3+\alpha_2$. Combining the last two equalities we obtain
$$2\alpha_1=\alpha_2+\alpha_4.$$
However, from (\ref{ord.first}) and (\ref{ord.third}) we have $2\alpha_1=3\alpha_2-\alpha_4$ which yields $\alpha_2=\alpha_4$ contradicting with $\alpha_4>\alpha_2$.

Now suppose that $\lambda_3+\alpha_1=\lambda_1+\alpha_3$. Then $\lambda_3>\lambda_1$. Comparing the third and fifth component of $v(t)$ we get $\lambda_2+\alpha_1=\lambda_2+\alpha_2$. The last two equations yield
\begin{equation}\label{lamb.1}
\lambda_3-\lambda_2=\alpha_3+\alpha_2-2\alpha_1.
\end{equation}
Comparing the fourth and fifth component of $v(t)$ we get
\begin{equation}\label{lamb.2}
\lambda_3-\lambda_2=\alpha_1-\alpha_2.
\end{equation}
Combining (\ref{lamb.1}), (\ref{lamb.2}) and (\ref{ord.first}) we get $\alpha_1=\alpha_3$, a contradiction.

Finally, suppose $\alpha_1=\alpha_2=\alpha_3=\alpha_4$. By similar considerations we get that the orders of the summands in each of the components of $v(t)$ are the same, so we reduce to the first case treated above.
\end{proof}
\begin{cor}\label{str.dep.} The column of partials with respect to $x_5$ is strictly dependent on the module generated by the remaining columns of $\mathrm{Jac}_z(X)$.
\end{cor}
Note that $X \cap H_{x_5}$ is the reduced union of $4$ coordinate axes in $\mathbb{C}^{4}$. Therefore, $\chi(X \cap H_{x_5})_s$=3 by the Example to Corollary 1.2.3 in \cite{BG}. The $8$-dimensional family (\ref{8-dim}) is induced by  an irreducible component $V_1$ of the miniversal base space of $(X,0)$. In fact, $V_1$ is the Artin component of $\mathrm{Def}(X,0)$ (see Theorem 3.2 in \cite{Wahl}). By setting $t_i=0$ for $i \neq 1$ in (\ref{8-dim}) we obtain a smoothing of $X$ in $V_1$. Denote the corresponding one-dimensional parameter space by $S_1$. In (\ref{4-dim}) we get another smoothing by setting $t_i=0$ for $i \neq 5$ that lies in a $4$-dimensional component  $V_2$ of the miniversal base space of $(X,0)$. Denote the corresponding parameter space for the smoothing by $S_2$. Finally, by abuse of notation denote by $M$ the relative Jacobian module defined on the total space of miniversal deformations of $(X,0)$.
\begin{Prop}\label{diff. lands.} We have
$$\mathrm{mult}_{S_1}\Gamma_2(M)=8 \ \ \text{and} \ \ \mathrm{mult}_{S_2}\Gamma_2(M)=6.$$
In particular, the Euler characteristics of the generic fibers of the deformations of $(X,0)$ over $S_1$ and $S_2$
are $11$ and $9$.
\end{Prop}
\begin{proof} As the computations for the two polar multiplicities are similar in nature, here we present only the computation of $\mathrm{mult}_{S_1}\Gamma_2(M)$. Denote by $\mathcal{X}_{S_1}$ the total space of the smoothing over $S_1$. Our first task is to find the equations of $\mathcal{R}(M)$ as a subvariety of $ \mathcal{X}_{S_1} \times \mathbb{P}^{4}$. Consider the map
\begin{equation}\label{Rees map}
\alpha : \mathcal{O}_{\mathcal{X}_{S_1}}[z_1,\ldots,z_5] \longrightarrow \mathcal{R}(M)
\end{equation}
where
$$z_i \rightarrow m_i$$
and the $m_i$ are the generators for $M$ corresponding to the columns of partials of the relative Jacobian matrix.
Finding the kernel of this map is computationally impossible even for software packages like Singular \cite{Singular}. However, we can reduce the complexity of this task by replacing $M$ with simpler module. Let $K$ denote the module generated by the columns of the $3 \times 5$ matrix obtained from $\mathrm{Jac}_z(\mathcal{X}_{S_1})$ whose specialization to $(X,0)$ is the matrix $[K]$ from (\ref{module K}). Observe that
$$\mathrm{Proj}(\mathcal{R}(M)) = \mathrm{Proj}(\mathcal{R}(K)).$$
Indeed, if $\alpha'$ denotes the map in (\ref{Rees map}) with $M$ replaced by $K$, then $\mathrm{Ker}(\alpha)=\mathrm{Ker}(\alpha')$ over the Zariski open dense subset $\mathrm{Spec}(\mathcal{O}_{\mathcal{X}_{S_1}})\setminus \mathbb{V}(x_5x_1)$.
Using Singular we find a minimal generating set for $\mathrm{Ker}(\alpha):$
\begin{displaymath}
\begin{split}
&(8x_4^2t+2x_2t^2+2t^3-2x_1x_4x_5^2-8x_3x_5^2t)z_1+(8x_1x_4t-2x_1^2x_5^2)z_2+\\
&(8x_2x_4t+6x_4t^{2}-4x_1x_2x_5^2+2x_3x_4x_5^2-2x_1x_5^2t)z_3+\\
&(8x_3x_4t+4x_1t^2-6x_1x_3x_5^2+4x_4^2x_5^2-2x_2x_5^2t+2x_5^2t^2)z_4+\\
&(3x_3x_5t+3x_4x_5^3)z_5,\\
&(16x_2x_4+4x_4t+2x_1x_5^2)z_1+(16x_3x_4-4x_2x_5^2)z_2+\\
&(16x_4^{2}-4x_2-10x_3x_5^2)z_3+(16x_1x_4-8x_3t)z_4-(4x_1x_5z_5+3x_5^3z_5)z_5, \\
&(8x_3x_4+2x_1t-8x_2x_5^2)z_1+(8x_4^2-8x_3x_5^2)z_2+(8x_1x_4-2x_3t)z_3+\\
&(8x_2x_4+4x_4t)z_4+(4x_2x_5+x_5t)z_5,\\
&(8x_2x_5+2x_5t)z_1+(6x_3x_5)z_2+(4x_4x_5)z_3+(2x_1x_5+2x_5^3)z_4-(4x_2+t)z_5,\\
&(2x_1x_5)z_1-(2x_3x_5)z_3-(4x_4x_5)z_4-(4x_1+3x_5^2)z_5.
\end{split}
\end{displaymath}
Let $M(0)$ be the image of $M$ in $F(0)$, where $F(0)$ is the restriction of the free module $F$ of rank $6$ that
contains $M$ to the closed fiber $(X,0)$. The map from $M$ to $M(0)$ induces a map on Rees algebras
$$\beta : \mathcal{R}(M) \rightarrow \mathcal{R}(M(0)).$$
Set $\delta = \beta \circ \alpha$. The map $\delta$ realizes $\mathrm{Proj}(\mathcal{R}(M(0)))$, which is the conormal space $C(X)$ of $(X,0)$, as a subspace of $\mathrm{Proj}(\mathcal{R}(M))$.
Let $J$ be the ideal in $\mathcal{O}_{\mathcal{X}_{S_1}}[z_1, \ldots, z_5]$ generated by $z_1+z_5,z_2+z_5,z_3,z_4+z_5$. Denote by $H_4$ the codimension $4$ hyperplane defined by $J$. Set $M_r := \alpha(J)$. We claim that $H_4$ does not intersect $C(X)$ after possibly replacing $(X,0)$ by a smaller representative. Indeed, by Corollary \ref{str.dep.}, the submodule $M_r(0)$ of $M(0)$ generated by the images of the the generators of $J$ under $\delta$ is a reduction of $M(0)$. Hence,
$M_r(0)M(0)^{q}=M(0)^{q+1}$ for $q \gg 0$, i.e. $\delta(J)$ is the irrelevant ideal in $\mathcal{R}(M(0))$.
The intersection of $H_4$ with the fiber of $\mathrm{Proj}(\mathcal{R}(M))$ over $0$ is $0$-dimensional scheme supported at $0 \in \mathcal{X}_{S_1}$. In fact,
$$ \mathcal{O}_{\mathcal{X}_{S_1}}[z_1,\ldots,z_5]/\langle \mathrm{Ker}(\alpha),J,t \rangle \cong \mathbb{C}\{1,x_3,x_4,x_5,x_3x_5,x_4x_5,x_5^2,x_5^3 \}.$$
Therefore, the degree of the component of maximal dimension of the fiber of $\mathrm{Proj}(\mathcal{R}(M))$ over $0$  is $8$ as $H_4$ misses all components of lower dimension. By conservation of number (see Proposition 10.2 in \cite{Fulton}) applied to the the proper map
$$\mathrm{Proj}(\mathcal{R}(M)) \cap H_4 \rightarrow S_1$$
we get $\mathrm{mult}_{S_1}\Gamma_2(M)=8$.
Alternatively, observe that $$\mathrm{Proj}(\mathcal{R}(M)) \cap H_4 = \mathrm{Proj}(\mathcal{R}(M)/M_r\mathcal{R}(M)).$$
Thus, $$\Gamma_2(M) = \mathrm{Supp}_{\mathcal{O}_{\mathcal{X}_{S_1}}}(M/M_r).$$
Next we claim that
\begin{equation}\label{polar red}
\mathrm{Supp}_{\mathcal{O}_{\mathcal{X}_{S_1}}}(M/M_r) =\mathbb{V}(\mathrm{Fitt}_{3}(F/M_r)),
\end{equation}
where $\mathrm{Fitt}_{3}(F/M_r)$ is the $3$rd Fitting ideal of $F/M_r$. Off $C(M)$, the non-free locus of $M$, the quotient $M/M_r$ is supported precisely at points where the rank of $M_r$ drops. Thus, $\mathrm{Supp}_{\mathcal{O}_{\mathcal{X}_{S_1}}}(M/M_r)$ is equal to the Zariski closure of $\mathbb{V}(\mathrm{Fitt}_{3}(F/M_r)) \setminus C(M).$ However, $M$ is free off $0 \in X$ because $X$ smoothens over $S_1$. Hence, $C(M)=0$, which proves (\ref{polar red}).

There is an alternative approach for computing $\mathrm{mult}_{S_1}\Gamma_2(M)$. Namely, using (\ref{polar red}) we can express $\mathrm{mult}_{S_1}\Gamma_2(M)$ as the difference
\begin{displaymath}
\dim_{\mathbb{C}} \mathbb{C}[x_1,\ldots,x_5]/\langle \mathrm{Fitt}_{3}(F/M_r),t-\epsilon,I \rangle - \dim_{\mathbb{C}} \mathbb{C}[x_1,\ldots,x_5]/I'
\end{displaymath}
where $\varepsilon$ is a generic constant, $I$ is the ideal that defines the total space of the deformation over $S_1$ and
$$I':= \bigcup_{i=0}^{\infty}(\langle \mathrm{Fitt}_{3}(F/M_r),t,I \rangle :_{\mathbb{C}[x_1,\ldots,x_5]} \mathfrak{m_{0}}^{i})$$
where $\mathfrak{m}_0$ is the maximal ideal generated by the $x_i$. The first term computes globally the degree of the covering  $\Gamma_2(M) \rightarrow S_1$, and the second term computes the number those branches that do not pass through the origin. One checks easily that $\dim_{\mathbb{C}} \mathbb{C}[x_1,\ldots,x_5]/\langle \mathrm{Fitt}_{3}(F/M_r),t-\epsilon,I \rangle =14$, $\dim_{\mathbb{C}} \mathbb{C}[x_1,\ldots,x_5]/I'=6$, and hence once again $\mathrm{mult}_{S_1}\Gamma_2(M)=14-6=8$.

A similar computation yields $\mathrm{mult}_{S_1}\Gamma_2(M)=6$. The part about the Euler characteristics follows at once by applying (\ref{polar-Euler}).
\end{proof}
Consider the family of singularities defined by the maximal minors of the following matrix
\begin{equation}
\begin{pmatrix}\label{W.8-dim}
x_1 & x_2 & x_3 &  x_4 \\
x_2 & x_3 & x_4 & x_1+x_5^{2}+tx_3
\end{pmatrix}.
\end{equation}
Denote by $Y_1$ the corresponding one-dimensional parameter space in $V_1$.
Another deformation of $(X,0)$, this time in $V_2$ is given by the vanishing of the $2$ by $2$ minors of the following matrix
\begin{equation}
\begin{pmatrix}\label{W.4-dim}
x_1 & x_2 & x_3 \\
x_2 & x_3 & x_4\\
x_3 & x_4 & x_1+x_5^{2}+tx_4
\end{pmatrix}.
\end{equation}
Denote by $Y_2$ the corresponding one-dimensional parameter space in $V_2$. Assume that the generic fibers of these two families deform to smooth fibers over $S_{i}'$ where $S_{i}'\subset V_i$. For a closed subspace $Y$ of the base space of miniversal deformations, denote by $\mathfrak{m}_Y$ the ideal of $Y$. Next we prove that the two families are Whitney equisingular, however, as the corresponding two landscapes of $(X,0)$ are distinct, the invariants of $(X,0)$ will be different.
\begin{Prop} We have
\begin{equation}\label{first fam}
\mathrm{mult}_{S_1}\Gamma_2(\mathfrak{m}_{S_1}M)=
\mathrm{mult}_{S_{1}'}\Gamma_2(\mathfrak{m}_{S_{1}'}M)=78
\end{equation}
and
\begin{equation}\label{second fam}
\mathrm{mult}_{S_2}\Gamma_2(\mathfrak{m}_{S_2}M)= \mathrm{mult}_{S_{2}'}\Gamma_2(\mathfrak{m}_{S_{2}'}M)=76.
\end{equation}
In particular, (\ref{W.8-dim}) and (\ref{W.4-dim}) are Whitney equisingular.
\end{Prop}
\begin{proof} First, note that the only singular points in the members of the two families lie on the corresponding parameter spaces. To prove that the two families are Whitney equisingular, or satisfy condition Whitney B, it suffices to show that $\mathrm{mult}_{Y_i}\Gamma_2(\mathfrak{m}_{Y_i}M)$ vanish for $i=1,2$ (see \cite{T-2}). By Theorem \ref{covering argument} applied with $Y=Y_i$, $Z_0=S_i$ and $Z_y=S_{i}'$, we have that $\mathrm{mult}_{Y_i}\Gamma_2(\mathfrak{m}_{Y_i}M)$ vanish if and only if the first equalities in (\ref{first fam}) and (\ref{second fam}) hold. Below we carry the computation for $\mathrm{mult}_{S_1}\Gamma_2(\mathfrak{m}_{S_1}M)$ only as the rest of the computations are performed in exactly the same manner. By the Multiplicity-Polar Theorem
\begin{equation}\label{two polars}
\mathrm{mult}_{S_1}\Gamma_2(\mathfrak{m}_{S_1}M) = e(\mathfrak{m}_{0}M(0),M(0))+\mathrm{mult}_{S_1}\Gamma_2(M)
\end{equation}
On the other hand, by Theorem 9.8 (i) from \cite{KT} it follows that
\begin{equation}\label{KT}
e(\mathfrak{m}_{0}M(0),M(0))= \binom{4}{1}m^{1}(X)+\binom{4}{2}m^{2}(X)
\end{equation}
where $m^{1}(X)$ is the multiplicity of the polar curve of $X$ and $m^{2}(X)$ is the multiplicity of $(X,0)$.
Since $(X,0)$ is a rational determinantal singularity defined by the vanishing of the maximal minors of a $2$ by $4$ matrix, it follows that $m^{2}(X)=4$ by \cite{Wahl}. Following the approach outlined at the end of the proof of Proposition \ref{diff. lands.} we can compute the polar curve of $(X,0)$ as the zero locus of the maximal minors of the submodule of $M$ obtained by taking $3$ general linear combinations of the generators for $M$. Indeed, the maximal minors define a determinantal variety, hence reduced, as the matrix is of the right size and $(X,0)$ is Cohen-Macaulay. An easy computations shows that $m^{1}(X)=9$.

 Thus,
$$e(\mathfrak{m}_{0}M(0),M(0))=4\cdot 4+6\cdot 9=70.$$
Hence, by Proposition \ref{diff. lands.} and (\ref{two polars}) it follows that
$$\mathrm{mult}_{S_1}\Gamma_2(\mathfrak{m}_{S_1}M)=70+8=78.$$
\end{proof}
Note that these computations can be done using the presentation matrix only as shown in Theorem 5.2.

 \end{document}